\documentclass{amsart}
\usepackage{amssymb,mathrsfs}
\usepackage{color,umoline}
\usepackage[dvipsnames]{xcolor}
\usepackage{graphicx}
\usepackage{enumitem}
\usepackage[font=small,labelfont=bf]{caption}
\usepackage{tikz, caption}
\usetikzlibrary{patterns}
\usepackage{relsize}
\usepackage{kotex}
\usepackage[mathscr]{eucal}
\usepackage{dsfont}
\usepackage{verbatim}
\usepackage[draft]{todonotes}
\usepackage[%colorlinks, 
citecolor=black,  %pagebackref,
hypertexnames=false]{hyperref}
\usepackage[dvipsnames]{xcolor}
\usepackage{subfig}
\usepackage{mathtools}

\usepackage{nicematrix}

\DeclareFontFamily{U}{mathx}{}
\DeclareFontShape{U}{mathx}{m}{n}{<-> mathx10}{}
\DeclareSymbolFont{mathx}{U}{mathx}{m}{n}
\DeclareMathAccent{\widecheck}{0}{mathx}{"71}

\def\wt#1{\widetilde{#1}}

\newcommand{\C}{\mathbb{C}}
\newcommand{\N}{\mathbb{N}}
\newcommand{\R}{\mathbb{R}}

\newcommand{\supp}{\operatorname{supp}}
\newcommand{\dist}{\operatorname{dist}}
\newcommand{\rank}{\operatorname{rank}}

\newtheorem{thm}{Theorem}[section]

\newtheorem{prop}[thm]{Proposition}

\newtheorem{lemma}[thm]{Lemma}
\numberwithin{equation}{section}

\def\fS{{\mathfrak {S}}}
\def\fT{{\mathfrak {T}}}

\def\bfb{{\mathbf {b}}}

\def\bfu{{\mathbf {u}}}
\def\bfv{{\mathbf {v}}}
\def\bfw{{\mathbf {w}}}

\def\cA{{\mathcal {A}}}

\def\cC{{\mathcal {C}}}
\def\cD{{\mathcal {D}}}

\def\cF{{\mathcal {F}}}
\def\cG{{\mathcal {G}}}
\def\cH{{\mathcal {H}}}
\def\cI{{\mathcal {I}}}

\def\cK{{\mathcal {K}}}

\def\cO{{\mathcal {O}}}

\def\cS{{\mathcal {S}}}
\def\cT{{\mathcal {T}}}

\newcommand{\ep}{{\varepsilon}}

\newcommand{\inp}[2]{\langle #1, #2\rangle}

\begin{document}

\author[Ryu]{Jaehyeon Ryu}
\address{Department of Mathematics, Ewha Womans University, Seoul
03760, Republic of Korea}
\email{jhryu67@ewha.ac.kr}

\title[Almost everywhere convergence]{On almost everywhere convergence of Bochner--Riesz means below the critical index}

\begin{abstract}
In this paper, we study the almost everywhere convergence problem for the Bochner--Riesz means $S_t^\delta f$ for $f\in L^p(\mathbb R^d)$ in the subcritical range
\[
    0\le \delta < \delta(d,p):=d\Big(\frac12-\frac1p\Big)-\frac12,
    \qquad \frac{2d}{d-1}<p<\infty,
\]
where $d\ge 2$. In this regime, the operator need not be well defined for fixed $t>0$, even as a tempered distribution. Nevertheless, the family $\{S_t^\delta f\}_{t>0}$ can still be interpreted as a distribution on $\mathbb R^d\times(0,\infty)$. We introduce an admissible class $\mathcal C_{p,\delta}\subset L^p(\mathbb R^d)$ on which this distribution has sufficient regularity in the $t$ variable to formulate the almost everywhere convergence problem. We establish three results concerning this class. First, we show that $\mathcal C_{p,\delta}\neq L^p(\mathbb R^d)$, and thus the almost everywhere convergence problem cannot be formulated for all $L^p$ functions below the critical index. Second, we show that this admissible class is nevertheless large in the sense that, for every $f\in L^p(\mathbb R^d)$, the pullback $f_V=f(V^{-1}\cdot)$ is admissible for Haar-a.e. volume-preserving upper triangular matrix $V$ with positive diagonal entries. Finally, we construct an $f\in \mathcal C_{p,\delta}$ for which $S_t^\delta f$ fails to converge almost everywhere as $t\to\infty$. A key ingredient in our argument is a multiparameter variant of the Bochner--Riesz means.
\end{abstract}

\maketitle

\section{Introduction}
We consider the Bochner--Riesz means on $\R^d$
\[
S_t^\delta f(x):=\frac{1}{(2\pi)^d}\int \Big(1-\frac{|\xi|^2}{t^2}\Big)_+^\delta e^{i\inp x \xi}\hat{f}(\xi) d\xi,\quad \delta\ge 0,\ \ t>0,\ \ f\in \cS(\R^d).
\]
In this paper, we study almost everywhere convergence of the Bochner--Riesz means for $L^p$ functions in the range $p_* := 2d/(d-1) < p < \infty$. 
A classical result of Carbery, Rubio de Francia, and Vega \cite{CRV88} asserts that, for all $f\in L^p(\R^d)$,
\[
\lim_{t\to\infty} S_t^\delta f(x) = f(x)\quad \text{for almost every (a.e.) $x\in \R^d$}
\]
provided that $p_* < p < \infty$ and $\delta > \delta(d,p)$, where $\delta(d,p)$ denotes the critical index
\[
\delta(d,p) = d(\frac12-\frac1p)-\frac12.
\]
For the endpoint $\delta = \delta(d,p)$, Lee and Seeger \cite{LS15} studied the convergence of variants of the Bochner--Riesz means and showed that, if $2d/(d-1)<p<\infty$, then $S_t^{\delta(d,p)}f$ converges to $f$ almost everywhere for $f\in L^{p,1}(\R^d)$.

The objective of this paper is to investigate how the almost everywhere convergence problem for the Bochner--Riesz means can be formulated for $f\in L^p$ in the subcritical range
\[
    0\le \delta < \delta(d,p),\quad p_*<p<\infty.
\]
In this regime, even formulating the problem becomes nontrivial. Indeed, the operator cannot in general be defined pointwise, or even as a tempered distribution for fixed $t$. This is due to the fact that the convolution kernel
\begin{align*}
    K^\delta(x-y) = \frac{1}{(2\pi)^d}\int \big(1-|\xi|^2\big)_+^\delta e^{i\inp {x-y}{\xi}} d\xi
\end{align*}
fails to belong to $L^{p'}(\R^d)$. We refer the reader to Carbery and Soria \cite{CaS88} and \cite{LS15} for related discussion. 

Nevertheless, to formulate the almost everywhere convergence problem, one need not define $S_t^\delta f$ separately for each fixed $t>0$. Indeed, the family $\{S_t^\delta f\}_{t>0}$ can still be interpreted as a distribution on $\mathbb R^d\times(0,\infty)$, and this observation is the starting point of this paper. Denote $\R_+ = (0,\infty)$. For $\delta\ge 0$ and $f\in L^p(\R^d)$, we then define a linear functional $S^\delta f$ on $C_c^\infty(\R^d\times \R_+)$ by
\begin{align*}
    S^\delta f (\varphi) = \int\bigg(\int t^dK^\delta(t(x-y)) \varphi(x,t) dxdt\bigg)f(y) dy,\quad \varphi\in C_c^\infty(\R^d\times \R_+).
\end{align*}
Using integration by parts, one can show that $S^\delta f(\varphi)$ is continuous in $\varphi$ for every $f\in L^p(\R^d)$. We refer to Section \ref{sec:thmxtU} below for details.
When $\delta > \delta(d,p)$, this definition agrees with the classical Bochner--Riesz means $S_t^\delta f(x)$ in the sense of distributions.

To formulate the almost everywhere convergence problem, it is natural to require that $S^\delta f$ be identified with a function of $(x,t)$ having enough regularity in the $t$ variable. For $p_* < p <\infty$ and $0\le \delta < \delta(d,p)$, we define
\begin{align*}
    \cC_{p,\delta} = \{f\in L^p(\R^d): S^\delta f\in L_{x,loc}^1(\R^d;H^s_{t,loc}(\R_+))\ \text{for some $s>\frac12$}\}.
\end{align*}
Here, $L_{x,loc}^1(\R^d;H^s_{t,loc}(\R_+))$ is the mixed-norm space defined by
\begin{align*}
    L_{x,loc}^1(\R^d;H^s_{t,loc}(\R_+)) = \{g\in \cD'(\R^d\times \R_+): \eta g\in L^1_x(H_t^s)\ \ \text{$\forall\eta\in C_c^\infty(\R^d\times \R_+)$}\},
\end{align*}
where $\cD'(\R^d\times \R_+)$ is the space of distributions on $\R^d\times \R_+$ and $H^s$ denotes the Sobolev space given by
\begin{align*}
    H^s(\R) = \{f\in \cS'(\R): (1+|\xi|^2)^{s/2}\hat f\in L^2(\R)\}.
\end{align*} 
%See \cite{T83, E10} for a discussion of Sobolev spaces.

We call $\cC_{p,\delta}$ the admissible class. By the Sobolev embedding, for $f\in \cC_{p,\delta}$, $S_t^\delta f(x)\equiv S^\delta f(x,t)$ admits a continuous representative in $t$ for a.e. $x\in \R^d$. %This makes it meaningful to ask whether almost everywhere convergence holds when $\delta$ is below the critical index. 
Thus, $\mathcal C_{p,\delta}$ provides a natural setting in which the almost everywhere convergence problem can be formulated below the critical index.

In this paper, we study the following three questions.
\begin{itemize}[topsep=1pt, itemsep=2pt]
    \item Does $\cC_{p,\delta}$ coincide with $L^p(\R^d)$, \textit{i.e.,} $\cC_{p,\delta} = L^p(\R^d)$?
    \item If not, how large is this admissible class?
    \item Does almost everywhere convergence hold for every $f\in \cC_{p,\delta}$?
\end{itemize}

Our first main result gives a negative answer to the first question.

\begin{thm}\label{thm:notequal}
    Let $I\subset \R_+$ be a compact interval with nonempty interior. Assume that $0\le \delta < \delta(d,p)$ and $p_* < p < \infty$. Then there exist $f\in L^p(\R^d)$ and $\eta\in C_c^\infty(\R^d\times \R_+)$ such that $\supp(\eta)\subset \R^d\times I$ and 
    \[
    \eta\,S^\delta f\notin L^1_x(\R^d;H^s_t(\R_+))
    \]
    for every $s\ge 1/2$.
\end{thm}

This theorem demonstrates that, within this framework, the almost everywhere convergence problem for the Bochner--Riesz means cannot be formulated for all $f\in L^p$ when $\delta$ is below the critical index. The proof of Theorem \ref{thm:notequal} is based on the uniform boundedness principle. We shall construct test functions $\varphi_n$, uniformly normalized in the dual space of $L_x^1(H_t^{1/2})$, for which the corresponding localized linear functionals on $L^p(\mathbb R^d)$ have unbounded operator norms. This will yield the failure of local $H_t^{1/2}$ regularity for a suitable $f\in L^p$.

We now turn to the second question concerning the size of $\mathcal C_{p,\delta}$, namely, how large the class of functions is for which the almost everywhere convergence problem can be formulated. An immediate observation is that $L^2(\R^d)\cap L^p(\R^d)\subset \cC_{p,\delta}$, which implies that $\cC_{p,\delta}$ is dense in $L^p(\R^d)$. %This indicates that $\cC_{p,\delta}$ is not small in the $L^p$ topology.

Our next result provides more explicit information on the size of $\mathcal C_{p,\delta}$ and shows that, although $\mathcal C_{p,\delta}\neq L^p(\R^d)$, this class is nevertheless large. To state the result, let $ST(d)_+$ denote the group of upper triangular real matrices $V$ whose diagonal entries are positive and the determinant is $1$. $ST(d)_+$ is a Lie group endowed with a left-invariant Haar measure, see Section \ref{sec:thmxtU} below.

\begin{thm}\label{thm:define-xt}
    Let $d\ge 2$, $p_* < p < \infty$, and $0\le \delta < \delta(d, p)$. Then, for every $f\in L^p(\R^d)$, we have $f_V\in \cC_{p,\delta}$ for Haar-a.e. $V\in ST(d)_+$.
\end{thm}

Here, $f_V$ denotes the pullback of $f$ by the linear map $y\mapsto V^{-1} y$ given by $f_V(y) = f(V^{-1} y)$. This theorem suggests that the admissible class is large in the following sense: although a given $L^p$ function need not be admissible, its pullback by Haar-a.e. element of $ST(d)_+$ is admissible. The proof revolves around a multiparameter variant of the Bochner--Riesz means, as described below.

Lastly, we consider whether $S^\delta f(x,t)$ converges to $f(x)$ as
$t\to\infty$ almost everywhere for all $f\in \mathcal C_{p,\delta}$. The
following result shows that this is not the case.

\begin{thm}\label{thm:failure}
Suppose that $d\ge 2$, $p_* < p <\infty$, and $0\le \delta < \delta(d,p)$. Then there exists a function $f\in \cC_{p,\delta}$ such that $S_t^\delta f(x)$ diverges as $t\to \infty$ on a set of positive measure.
\end{thm}

To construct the desired counterexample, we follow approaches in \cite{St61, JLR22}. The main task is to construct a sequence of $L^p$-normalized functions $\{f_j\}_{j\ge 1}$ such that $\sup_{t>0} |S_t^\delta f_j(x)|$ tends to infinity as $j\to \infty$ for $x$ in a set of positive measure. The desired example function $f$ is then given by
        \[
        f(x) = \sum_{j\ge 1} a_j f_j
        \]
where $\{a_j\}_{j\ge 1}$ is a rapidly decreasing sequence of non-negative real numbers chosen so that the series converges in $L^p(\R^d)$.

The functions $f_j$ are constructed explicitly. Specifically, each $f_j$ is constructed as the kernel of a smoothed spectral projection for $-\Delta$ localized to a window of width $\ep_jN_j$ around $N_j$, where $\ep_j\to 0$, $N_j\to\infty$ as $j\to \infty$.

\subsubsection*{Multiparameter variant of Bochner--Riesz means}
A key ingredient in the proof of Theorem \ref{thm:define-xt} is the introduction of a multiparameter variant of the Bochner--Riesz means. To this end, we consider the group of upper triangular real matrices $T(d)_+$ with positive diagonal entries. It is known that $T(d)_+$ is a Lie group. 
Throughout this paper, we write $U=tV$ with $(t, V)\in \R_+\times ST(d)_+$. Let $dU$ denote a left-invariant Haar measure on $T(d)_+$.
% the Haar measure $dU$ is expressed as
%\[
%dU = \frac{dt}{t}dV.
%\]
%See \cite{F08} for details on the Haar measure.

For $f\in L^p(\R^d)$ and $\wt\varphi \in C_c^\infty(\R^d\times T(d)_+)$, we define
\begin{align*}
    \mathcal T^\delta f(\wt\varphi) =  \int\bigg(\int t^d K^\delta(tV(x-y)) \wt\varphi(x,tV)  dxdU\bigg)f(y) dy.
\end{align*}

\begin{thm}\label{thm:define-xtU}
    Let $d\ge 2$, $p_* < p < \infty$, and $0\le \delta < \delta(d,p)$. Then, for every $f\in L^p(\R^d)$,
    \[
    \cT^\delta f\in L_{x,V,loc}^2(\R^d\times ST(d)_+;H^s_{t,loc}(\R_+))
    \]
    with some $s>1/2$.
\end{thm}
 
An advantage of introducing this multiparameter variant $\cT^\delta$ becomes apparent when the kernel $K^\delta(tV(x-y))$ is expressed via an asymptotic expansion involving the oscillatory term $e^{it|V(x-y)|}$. While the phase function $\Phi(tV,x-y) = t|V(x-y)|$ has a degenerate mixed Hessian in the variables $(x,t)$ and $y$ with $V$ fixed, it becomes nondegenerate away from the diagonal $x=y$ when $V$ is allowed to vary; see Proposition \ref{prop:rank} below. This improved oscillatory structure enables us to establish the desired regularity of $\cT^\delta f$ in $t$, as the regularity is governed by the oscillatory behavior of the defining integral.

To derive Theorem \ref{thm:define-xt} from Theorem \ref{thm:define-xtU}, we make use of the scaling property of $\cT^\delta f$, which formally yields the identity
\[
\cT^\delta f_V(x,tI) = \cT^\delta f(V^{-1}x,tV),
\]
where $f_V(y) = f(V^{-1}y)$ and $I$ is the identity matrix. This identity is to be interpreted in the sense of distributions (see Lemma \ref{lem:identity1}). Despite its simplicity, this identity turns out to play a crucial role in deriving Theorem \ref{thm:define-xt}. %Precisely, this identity allows us to deduce that $S^\delta f$ belongs to the mixed-norm space in Theorem \ref{thm:define-xt} for all $f\in \cC_{p,\delta}$. 

\subsubsection*{Organization} We begin by proving Theorem \ref{thm:notequal} in Section \ref{sec:notequal}. Section \ref{sec:thmxtU} is devoted to the proof of Theorem \ref{thm:define-xtU}. In Section \ref{sec:implication}, we deduce Theorem \ref{thm:define-xt} from Theorem \ref{thm:define-xtU}. Finally, in Section \ref{sec:failure}, we prove Theorem \ref{thm:failure}, which establishes the failure of almost everywhere convergence of the Bochner--Riesz means for $\delta < \delta(d,p)$.

\subsubsection*{Notation}
For nonnegative quantities $A$ and $B$ and parameters $L$, we write $A\lesssim_L B$ if $A\le CB$ for some $C>0$ depending on $L$, and $A\sim_L B$ if both $A\lesssim_L B$, $B\lesssim_L A$ hold. We write $A\ll_L B$ if $C$ can be chosen sufficiently small. If the constant $C$ depends only on $d,p$, we suppress this dependence and write $A\lesssim B$ or $A\sim B$. To lighten the notation, we omit domains in mixed norm spaces and write $L_{x,loc}^2(H_{t,loc}^s)$, $L^2_{x,V,loc}(H_{t,loc}^s)$ when they are clear from the context. Similarly, we simply write $L^p_{x,t} = L^p(\R^d\times \R_+)$ and $L^p_{x,t,V} = L^p(\R^d\times \R_+\times ST(d)_+)$.

\section{Proof of Theorem \ref{thm:notequal}}\label{sec:notequal}

We begin with the proof of Theorem \ref{thm:notequal}, which states that $\cC_{p,\delta}\neq L^p(\R^d)$. In the proof of the theorem, we assume $I$ is a compact interval in $\R_+$. Since $H^s(\R)$ is a subspace of $H^{1/2}(\R)$ for $s>1/2$, it is sufficient to prove Theorem \ref{thm:notequal} with $s = 1/2$, that is, the existence of $f\in L^p$ and $\eta\in C_c^\infty(\R^d\times \R_+)$ satisfying
\begin{align}\label{r:notin-1/2}
    \eta\,S^\delta f\notin L_x^1(H_t^{\frac12}).
\end{align}
To reduce the problem, we use the uniform boundedness principle. We first show the following lemma, which follows easily from integration by parts.

\begin{lemma}\label{lem:integrable}
    Let $0\le \delta <\delta(d,p)$ and let $\breve\varphi\in C_c^\infty(\R_+)$. Then, for $N\in \N$, there exists a constant $C_{\breve\varphi, N}>0$ such that
    \begin{align*}
        \bigg|\int t^d K^\delta (t(x-y)) \breve\varphi(t) dt \bigg|\le C_{\breve\varphi, N}|x-y|^{-\frac{d+1}{2}-\delta-N} \|\breve \varphi\|_{C^N}
    \end{align*}
    for $x,y\in \R^d$ such that $|x-y|\ge 1$. Here, $\|\breve \varphi\|_{C^N}$ denotes the supremum norm on the space $C^N(\R)$ defined by $\|\breve \varphi\|_{C^N} := \max\limits_{0\le k\le N} \sup\limits_t|\partial_t^k\breve\varphi(t)|$.
\end{lemma}

\begin{proof}
    Assume that $x,y\in \R^d$ such that $|x-y|\ge 1$. By the assumption $|x-y|\ge 1$, we may assume that $t|x-y|\ge c$ for $t\in \supp(\breve\varphi)$. Hence, by \cite[p.390]{St93},
    \begin{align}
    \begin{aligned}\label{i:asymKdel}
        K^\delta (t(x-y)) =  t^{-\frac{d+1}{2}-\delta}|x-y|^{-\frac{d+1}{2}-\delta} \sum_{\pm}e^{\pm it|x-y|}&A_\pm(t|x-y|)  \\
        &\qquad+ E_M(t|x-y|),
    \end{aligned}
    \end{align}
    where $A_+, A_-$ are defined by
    \begin{align*}
        A_\pm(r) = \sum_{j=0}^{M-1} a_{\pm, j}\,r^{-j},\quad r>0,
    \end{align*}
    with constants $a_{\pm, j}\in \C$ such that $a_{\pm,0}\neq 0$. $E_M$ is a smooth function on $\R_+$ satisfying 
    \begin{align}\label{e:EMbd}
        |E_M(r)| + |E_M'(r)|\le C_Mr^{-\frac{d+1}{2}-\delta-M}
    \end{align}
    for $r\ge c$ with a constant $c>0$, where $E_M'$ denotes the derivative of $E_M$. By repeated integration by parts based on the operator $(i|x-y|)^{-1}\partial_t$, we have
    \begin{align*}
        \bigg|\int t^dK^\delta (t(x-y)) \breve\varphi(t) dt \bigg| &\le |x-y|^{-\frac{d+1}{2}-\delta-N} \int \big|\partial_t^N\big(t^{\frac{d-1}{2}-\delta} A_\pm(t|x-y|)\breve\varphi(t)\big) \big| dt \\
        & \quad + |x-y|^{-\frac{d+1}{2}-\delta-M}\|\breve \varphi\|_{L^1} \\
        &\le C_{\breve\varphi, N} |x-y|^{-\frac{d+1}{2}-\delta-N} \|\breve\varphi\|_{C^N}
    \end{align*}
    with a constant $C_{\breve\varphi, N}>0$ for $N<M$. This proves the lemma.
\end{proof}

We continue the proof of Theorem \ref{thm:notequal}. For
$\eta, \varphi \in C_c^\infty(\mathbb R^d \times \mathbb R_+)$, define
\begin{align*}
    \fS^\delta_{\eta}[\varphi](f)
    = \Lambda(\eta S^\delta f)(\varphi),
\end{align*}
where $\Lambda = (1-\partial_t^2)^{1/4}$. When $\Lambda$ acts on $\varphi$, we regard $\varphi$ as a compactly supported smooth function of $t$ on $\mathbb R$ by extending it by zero outside $\R_+$. We write
\begin{align}\label{i:cGdeta-ker}
    \fS^\delta_{\eta}[\varphi](f)
    = \int \Big(\int \Big(\int t^d K^\delta(t(x-y))
    \eta(x,t) \Lambda \varphi(x,t)\, dt\Big) dx\Big) f(y)\,dy .
\end{align}
Applying Lemma \ref{lem:integrable} to the right-hand side gives
\begin{align*}
    |\fS^\delta_{\eta}[\varphi](f)|
    \le C \int (1+|x-y|)^{-\frac{d+1}{2}-\delta-N}
    \|\eta(x,\cdot)\Lambda\varphi(x,\cdot)\|_{C^N(\mathbb R)}
    |f(y)|\,dxdy,
\end{align*}
where $C=C(N,\eta,\varphi)>0$. Here, the contribution from the region $|x-y|<1$ is harmless, since $K^\delta$ is bounded and the $t$-support of $\eta$ is compact. Choose $N\in\N$ so large that
\[
    \frac{d+1}{2}+\delta+N > \frac{d}{p'}.
\]
Since $\|\eta(x,\cdot)\Lambda\varphi(x,\cdot)\|_{C^N(\mathbb R)}$ is compactly supported as a function of $x$, Hölder's inequality yields
\begin{align*}
    |\fS^\delta_{\eta}[\varphi](f)|
    \le C \|f\|_{L^p(\mathbb R^d)}
\end{align*}
for some constant $C=C(\eta,\varphi)>0$. Thus, for every
$\eta,\varphi\in C_c^\infty(\mathbb R^d\times\mathbb R_+)$,
$\fS^\delta_{\eta}[\varphi]$ is a bounded linear functional on
$L^p(\mathbb R^d)$.

Suppose now that there exist $\eta\in C_c^\infty(\mathbb R^d\times\mathbb R_+)$
and a sequence
$\{\varphi_n\}_{n\in\N}\subset C_c^\infty(\mathbb R^d\times\mathbb R_+)$
such that
\begin{align}\label{p:phinpro1}
    \|\varphi_n\|_{L^\infty_x(\mathbb R^d;L^2_t(\mathbb R_+))}=1
\end{align}
for every $n\in\N$, and
\begin{align}\label{p:phinpro2}
    \sup_n \|\fS^\delta_{\eta}[\varphi_n]\|_{L^p(\mathbb R^d)\to \mathbb C}
    = \infty .
\end{align}
Since each $\fS^\delta_{\eta}[\varphi_n]$ is a bounded linear functional on
$L^p(\mathbb R^d)$, the uniform boundedness principle gives an
$f\in L^p(\mathbb R^d)$ such that
\begin{align*}
    \sup_n |\fS^\delta_\eta[\varphi_n](f)| = \infty .
\end{align*}
This implies \eqref{r:notin-1/2}. Indeed, suppose for contradiction that
$\eta S^\delta f\in L^1_x(H_t^{1/2})$. Then, by Hölder's inequality and
\eqref{p:phinpro1},
\begin{align*}
    |\Lambda(\eta S^\delta f)(\varphi_n)|
    \le \|\eta S^\delta f\|_{L^1_x(H_t^{1/2})},
\end{align*}
uniformly in $n$. Hence $\fS^\delta_\eta[\varphi_n](f)$ is uniformly bounded
in $n$, which contradicts \eqref{p:phinpro2}.

Thus, the task is reduced to the problem of constructing $\eta$ and $\{\varphi_n\}$ such that \eqref{p:phinpro1} and \eqref{p:phinpro2} hold. Let $\ep>0$ be a small constant. Choose $\eta_\circ,\varphi_\circ\in C_c^\infty(\mathbb R^d)$ such
that
\[
    \supp(\varphi_\circ)\subset B(0,\ep),\quad
    \supp(\eta_\circ)\subset B(0,2\ep),
\]
and $\eta_\circ(x)=1$ for $x\in B(0,\ep)$. Moreover, assume that $\|\varphi_\circ\|_{L^\infty} = 1$ and $0\le \varphi_\circ(x)\le 1$ for all $x\in \R^d$. By definition,
\begin{align}\label{i:vpcirceta}
    \eta_\circ(x)\varphi_\circ(x) = \varphi_\circ(x)\quad \text{for all}\ x\in \R^d.
\end{align}
Denote $I = [t_I - \ell_I, t_I+\ell_I]$ for $t_I\in \R_+$ and $\ell_I>0$. We define $\eta$ and $\varphi_n$ on
$\mathbb R^d\times \mathbb R_+$ by
\[
    \eta(x,t)=t^{\delta-\frac{d-1}{2}}\eta_\circ(x)\eta_*(t-t_I),
    \quad
    \varphi_n(x,t)=\varphi_\circ(x)\varphi_{n,*}(t-t_I),
\]
where $\eta_*,\varphi_{n,*}\in C_c^\infty(\mathbb R)$ will be chosen so that
\[
    \supp(\eta_*),\ \supp(\varphi_{n,*})\subset (-\ell_I,\ell_I).
\]
Furthermore, for each $n\in\N$, set
\[
    f_n(y)=e^{-it_I|y|}\chi_{A_n}(y),
\]
where
\[
    A_n=\{y\in\mathbb R^d: B_I2^{n-\frac12}\le |y|\le B_I2^{n+\frac12}\}.
\]
Here $B_I>0$ is chosen so that $B_I>8(1+\ell_I^{-1})$. If $\ep>0$ is
sufficiently small, then
\begin{align}\label{r:range|x-y|}
    |x-y|\in [B_I2^{n-1},B_I2^{n+1}]
\end{align}
for all $(x,y)\in \supp(\varphi_\circ)\times A_n$.

Fix $n\in\N$ for the moment. Recall \eqref{i:cGdeta-ker}. Using \eqref{i:asymKdel}, and making the
change of variables $t\mapsto t+t_I$ in the $t$-integral, we write
\begin{align}\label{i:fSvphifn}
    \fS^\delta_\eta[\varphi_n](f_n)=I_{n,+}+I_{n,-}+I\!I_n,
\end{align}
where
\begin{align*}
    I_{n,\pm}
    &= a_{\pm,0}
    \int \Big(\int e^{\pm i(t+t_I)|x-y|}
    \eta_*(t)\Lambda\varphi_{n,*}(t)\,dt\Big)
    |x-y|^{-\frac{d+1}{2}-\delta}
    \varphi_\circ(x)f_n(y)\,dxdy, \\
    I\!I_n
    &= \int \Big(\int E_1((t+t_I)|x-y|)
    (t+t_I)^{\frac{d+1}{2}+\delta}
    \eta_*(t)\Lambda\varphi_{n,*}(t)\,dt\Big)
    \varphi_\circ(x)f_n(y)\,dxdy .
\end{align*}
Here we use \eqref{i:vpcirceta}, \eqref{r:range|x-y|}, and that $\Lambda$ is translation-invariant in the $t$-variable. The term $I_{n,+}$ is the main contribution, because the phase $e^{it_I|x-y|}$ is canceled by the oscillation of $f_n$ up to a harmless error on $\supp(\varphi_\circ)\times A_n$. The terms $I_{n,-}$ and $I\!I_n$ correspond to the error terms.

To estimate $I_{n,+}$, we note
\begin{align}\label{i:Feta*Laphi}
    \frac{1}{2\pi}\int e^{it|x-y|}\eta_*(t)\Lambda\varphi_{n,*}(t)\,dt
    =
    \big(\mathcal F^{-1}(\eta_*) *
    \mathcal F^{-1}(\Lambda\varphi_{n,*})\big)(|x-y|),
\end{align}
where $\mathcal F^{-1}$ denotes the inverse Fourier transform defined by
\[
\cF^{-1}(\eta_*) (s) = \frac{1}{2\pi} \int \eta_*(t) e^{its} dt.  
\]
We choose $\eta_*$ so that $\supp(\eta_*)\subset (-\ell_I,\ell_I)$ and
$\mathcal F^{-1}\eta_*\ge 0$ on $\mathbb R$. Such a function may be obtained
by taking
\[
    \eta_*=\wt\eta_* * \overline{\wt\eta_*(-\cdot)}
\]
for some $\wt\eta_*\in C_c^\infty(\mathbb R)$ with
$\supp(\wt\eta_*)\subset (-\ell_I/2,\ell_I/2)$.

Next, choose $\wt\varphi_{n,*}\in \mathcal S(\mathbb R)$ such that
\[
    0\le \mathcal F^{-1}(\wt\varphi_{n,*})\le 1,\quad \supp(\mathcal F^{-1}(\wt\varphi_{n,*}))
    \subset (B_I2^{n-3},B_I2^{n+3}),
\]
and
\[
    \mathcal F^{-1}(\wt\varphi_{n,*})(s)=1
    \quad
    \text{for } s\in [B_I2^{n-2},B_I2^{n+2}].
\]
Define
\[
    \mathring\varphi_{n,*}(t)
    =\eta_*(t)\wt\varphi_{n,*}(t),
    \qquad
    \varphi_{n,*}(t)
    =
    \frac{\mathring\varphi_{n,*}(t)}
    {\|\mathring\varphi_{n,*}\|_{L^2}}.
\]
Clearly, $\mathring\varphi_{n,*}\in C_c^\infty(\mathbb R)$ and
\[
    \supp(\mathring\varphi_{n,*})\subset (-\ell_I,\ell_I),\quad \|\varphi_{n,*}\|_{L^2} = 1.
\]
We claim that
\begin{align}\label{e:lbF-1vphi}
    \cF^{-1}(\mathring\varphi_{n,*})(s)\gtrsim 1,\,\ \  2^{-\frac n2}(\cF^{-1}(\eta_*)*\cF^{-1}({\Lambda\mathring\varphi_{n,*}}))(s)\gtrsim 1
\end{align}
if $s\in [B_I 2^{n-1}, B_I 2^{n+1}]$, and
\begin{align}\label{e:ubF-1vphi}
    |\cF^{-1}(\mathring\varphi_{n,*})(s)| + 2^{-\frac n2}|(\cF^{-1}(\eta_*)*\cF^{-1}({\Lambda\mathring\varphi_{n,*}}))(s)|\le C_N 2^{-Nn}
\end{align}
for every $N\in\N$, whenever
$s\in \mathbb R\setminus (B_I2^{n-4},B_I2^{n+4})$. These estimates imply
\begin{align}\label{e:sizeL2vphi}
    \|\mathring\varphi_{n,*}\|_{L^2}
    = (2\pi)^{\frac12}
    \|\mathcal F^{-1}(\mathring\varphi_{n,*})\|_{L^2}
    \sim 2^{\frac n2}.
\end{align}

To verify the claim, we prove them only for $\cF^{-1}(\mathring\varphi_{n,*})$, because the estimates for $\cF^{-1}(\eta_*)*\cF^{-1}({\Lambda\mathring\varphi_{n,*}})$ follow from the same argument. We note
\begin{align*}
    \cF^{-1}(\mathring\varphi_{n,*})(s) = \int \cF^{-1}(\eta_*)(s-v) \cF^{-1}(\wt\varphi_{n,*})(v) dv.
\end{align*}
By the choices of $\eta_*$ and $\wt\varphi_{n,*}$, this gives
\[
    \int_{B_I2^{n-2}}^{B_I2^{n+2}}
    \mathcal F^{-1}\eta_*(s-v)\,dv
    \le
    |\mathcal F^{-1}(\mathring\varphi_{n,*})(s)|
    \le
    \int_{B_I2^{n-3}}^{B_I2^{n+3}}
    \mathcal F^{-1}\eta_*(s-v)\,dv .
\]
From this, it follows that
\begin{align*}
    |\cF^{-1}(\mathring \varphi_{n,*})(s)|\ge \int_{-B_I2^{n-2}}^{B_I 2^{n-2}} \cF^{-1}(\eta_*)(v) dv 
\end{align*}
if $s\in [B_I2^{n-1}, B_I2^{n+1}]$, and
\begin{align*}
    |\cF^{-1}(\mathring \varphi_{n,*})(s)|&\le C_N \int_{B_I2^{n-3}}^{B_I2^{n+3}} (1+\ell_I|s-v|)^{-N} dv \\
    &\le C_N2^{-Nn}
\end{align*}
if $s\in \R\setminus (B_I2^{n-4}, B_I2^{n+4})$. This proves the desired bounds.

We now turn to the estimation of $I_{n,+}$. Substituting \eqref{i:Feta*Laphi} into the definition of $I_{n,+}$, we obtain
\begin{align*}
    I_{n,+}
    =& \,2\pi a_{+,0}\int_{A_n}\int e^{it_I(|x-y|-|y|)}
    |x-y|^{-\frac{d+1}{2}-\delta}
    \varphi_\circ(x) \\
    &\qquad\qquad\qquad \times
    \big(\mathcal F^{-1}(\eta_*)*
    \mathcal F^{-1}(\Lambda\varphi_{n,*})\big)(|x-y|)\,dxdy .
\end{align*}
By expanding $|x-y|$ in $x$, we have
\[
    t_I\big||x-y|-|y|\big|\le \pi/4
\]
for $(x,y)\in \supp(\eta_\circ)\times A_n$, provided that $\ep>0$ is sufficiently small. Combining this with
\eqref{e:lbF-1vphi} and \eqref{e:sizeL2vphi}, we get
\begin{align}\label{e:lbIn+}
    |I_{n,+}|
    &\gtrsim
    2^{-n(\frac{d+1}{2}+\delta)}
    \int \varphi_\circ(x)\chi_{A_n}(y)\,dxdy  \notag\\
    &\gtrsim
    2^{n(\frac{d-1}{2}-\delta)} .
\end{align}
We also use that $|x-y|\sim 2^n$ and $|A_n|\sim 2^{nd}$. On the other hand, by \eqref{e:ubF-1vphi},
\begin{align}
    |I_{n,-}|
    &\lesssim
    2^{-n(\frac{d+1}{2}+\delta)}
    \int
    \big|
    \big(\mathcal F^{-1}(\eta_*)*
    \mathcal F^{-1}(\Lambda\varphi_{n,*})\big)(-|x-y|)
    \big|
    \chi_{A_n}(y)\,dxdy \notag\\
    &\lesssim_N
    2^{n(\frac{d-1}{2}-\delta-N)} .
    \label{e:In-}
\end{align}
Furthermore, using $|E_1(t|x-y|)|\lesssim |x-y|^{-\frac{d+3}{2}-\delta}$ and Hölder's inequality, we estimate
\begin{align*}
    |I\!I_n|
    &\lesssim
    2^{-n(\frac{d+3}{2}+\delta)}
    \|\eta_*\Lambda\varphi_{n,*}\|_{L^1(\mathbb R)}
    \int \varphi_\circ(x)\chi_{A_n}(y)\,dxdy \\
    &\lesssim
    2^{n(\frac{d-3}{2}-\delta)}
    \|\mathring\varphi_{n,*}\|_{L^2}^{-1}
    \|\eta_*\Lambda\mathring\varphi_{n,*}\|_{L^2(\mathbb R)} .
\end{align*}
By Plancherel's identity, together with \eqref{e:ubF-1vphi} and
\eqref{e:sizeL2vphi}, this gives
\begin{align*}
    |I\!I_n|
    \lesssim
    2^{n(\frac{d-2}{2}-\delta)} .
\end{align*}
Substituting this estimate and \eqref{e:lbIn+}, \eqref{e:In-} into
\eqref{i:fSvphifn}, we obtain
\begin{align*}
    |\fS_\eta^\delta[\varphi_n](f_n)|
    \gtrsim
    2^{n(\frac{d-1}{2}-\delta)}
    - O\big(2^{n(\frac{d-2}{2}-\delta)}\big).
\end{align*}
Thus, for all sufficiently large $n$,
\begin{align*}
    |\fS_\eta^\delta[\varphi_n](f_n)|
    \gtrsim
    2^{n(\frac{d-1}{2}-\delta)} .
\end{align*}
Since $\|f_n\|_{L^p}\sim 2^{\frac{nd}{p}}$, it follows that
\begin{align*}
    \|\fS_\eta^\delta[\varphi_n]\|_{L^p\to \mathbb C}
    \gtrsim
    2^{n(\frac{d-1}{2}-\delta-\frac{d}{p})}.
\end{align*}
By the assumption $\delta<\delta(d,p)$, the exponent on the right-hand side is
positive. Thus, \eqref{p:phinpro2} holds. Furthermore, by construction, $\varphi_n$ satisfies \eqref{p:phinpro1}. This completes the proof.

\section{Proof of Theorem \ref{thm:define-xtU}}\label{sec:thmxtU}
In this section, we prove Theorem \ref{thm:define-xtU}. Before proceeding, we justify the existence of left-invariant Haar measures on $T(d)_+$ and $ST(d)_+$. Note that $T(d)_+$ and $ST(d)_+$ are closed subgroups of the general linear group $GL(d)$. Thus, by the closed subgroup theorem, these groups are Lie groups. In particular, they are locally compact. Hence, they admit left-invariant Haar measures. In the remainder of this paper, we fix left-invariant Haar measures on $T(d)_+$ and $ST(d)_+$, denoted by $dU$ and $dV$, respectively, such that
\[
dU = \frac{dt}{t}dV,\quad \text{$U = tV$ with $t\in \R_+$ and $V\in ST(d)_+$.}
\]
See \cite{F08} for details on the Haar measure.
Let $\chi_\circ \in C_c^\infty(\R^d)$ such that $\supp(\chi_\circ)\subset B(0,2)$ and $\chi_\circ(y) = 1$ for $y\in B(0,1)$. Define $\wt\chi\in C_c^\infty(\R^d)$ by 
\[
\wt\chi(y) = \chi_\circ(y) - \chi_\circ(2y).
\]
Moreover, for each $k\in \N$, we define
\[
\chi_{\circ, k}(y) = \chi_\circ(2^{-k}y), \quad \wt\chi_k(y) = \wt\chi(2^{-k}y).
\]
For $0\le \delta <\delta(d,p)$ and $k\ge 1$, we define $\cG_k^\delta f$ by
\begin{align*}
    \cG^\delta_k f(x,U) = \int t^d K^\delta (tV(x-y))\chi_\circ(2^{-k}y) f(y) dy.
\end{align*}
The following is the main ingredient in the proof.

\begin{prop}\label{prop:convergence}
    Let $\eta\in C_c^\infty(\R^d\times T(d)_+)$. Let $p,\delta$ be as in Theorem \ref{thm:define-xtU}. Then there exists $s>1/2$ such that $\{\eta\,\cG^\delta_k f\}_{k\ge 1}$ is Cauchy in $L_{x,V}^2(H^s_t)$ for all $f\in L^p(\R^d)$.
\end{prop}

For the moment, we assume Proposition \ref{prop:convergence} and proceed with the proof of Theorem \ref{thm:define-xtU}. It follows from Lemma \ref{lem:integrable} that, for every $\wt\varphi\in C_c^\infty(\R^d\times T(d)_+)$,
\begin{align}\label{e:Kdelp'}
    \int\bigg|\int t^d K^\delta(tV(x-y)) \wt\varphi(x,tV) dxdU\bigg|^{p'} dy \le C_{p', \wt\varphi} \|\wt\varphi\|_{C^N}
\end{align}
with a suitable $N\in\N$. Combining this with H\"older's inequality, we deduce that the linear functional $\wt\varphi\mapsto \cT^\delta f(\wt\varphi)$ is continuous on $C_c^\infty(\R^d\times T(d)_+)$. Hence, $\cT^\delta f\in \cD'(\R^d\times T(d)_+)$.

To prove $\cT^\delta f\in L^2_{x,V,loc}(H^s_{t,loc})$, we note that Proposition \ref{prop:convergence} is equivalent to the assertion that $\{\cG_k^\delta f\}_{k\ge 1}$ is Cauchy in $L^2_{x,V,loc}(H^s_{t,loc})$ for all $f\in L^p(\R^d)$. Since $L^2_{x,V,loc}(H^s_{t,loc})$ is a Fr\'echet space (see \cite{G09, Sch58}), there exists $g_f\in L^2_{x,V,loc}(H^s_{t,loc})$ such that $\cG_k^\delta f$ converges to $g_f$ in $L^2_{x,V,loc}(H^s_{t,loc})$ as $k$ tends to infinity, which implies
\begin{align}\label{i:limcT1}
    \int g_f(x,U) \wt \varphi(x,U) dxdU = \lim_{k\to\infty} \int  \cG^\delta_k f(x,U) \wt\varphi(x,U) dxdU
\end{align}
for all $f\in L^p(\R^d)$ and $\wt\varphi\in C_c^\infty(\R^d\times T(d)_+)$.
On the other hand, by \eqref{e:Kdelp'},
\begin{align*}
    &\lim_{k\to\infty} \int  \cG^\delta_k f(x,U) \wt\varphi(x,U) dxdU\\
    &\hspace{32mm} = \lim_{k\to \infty} \int \bigg(\int t^d K^\delta(tV(x-y)) \wt\varphi(x,U) dxdU\bigg) \chi_{\circ, k}(y) f(y) dy \\
    &\hspace{32mm} = \int \bigg(\int t^d K^\delta(tV(x-y)) \wt\varphi(x,U) dxdU\bigg) f(y) dy
\end{align*}
for every $\wt\varphi\in C_c^\infty(\R^d\times T(d)_+)$. Combining this with \eqref{i:limcT1}, we conclude that $\cT^\delta f = g_f$ as distributions, and hence $\cT^\delta f \in L^2_{x,V,loc}(H^s_{t,loc})$. This completes the proof of Theorem \ref{thm:define-xtU}.

\subsection{Cauchy property of $\cG_k^\delta f$}
In this subsection, we prove Proposition \ref{prop:convergence}. We fix a cutoff function $\eta\in C_c^\infty(\R^d\times T(d)_+)$ throughout the proof. 
Let $\wt\cG^\delta_k$ be an operator given by
\begin{align*}
    \wt\cG^\delta_k f(x,U) &= \cG^\delta_{k} f(x,U) - \cG^\delta_{k-1} f(x,U) \\
    &= \int t^d K^\delta(tV(x-y)) \wt\chi_k(y) f(y) dy.
\end{align*}
We shall prove that there exists a large natural number $k_0$ such that for all $k > k_0$ and $n\in \{0,1\}$,
\begin{align}\label{e:wtGdelk-weighted}
    \|\partial_t^n(\eta\,\wt\cG^\delta_k f)\|_{L_{x,t,V}^2} \le C2^{- k(\frac{d+1}{2}-n+\delta-\frac{d}{2p'})}\|f\|_{L^p(\R^d)}.
\end{align}
Once this estimate is established, the desired statement % that $\{\eta\,\cG^\delta_{k,0,+} f\}$ is Cauchy in $L_{x,V}^2(H^s_t(\R_+))$ 
rather immediately follows. Indeed, by an interpolation inequality in Sobolev spaces (see \cite{BM18}), \eqref{e:wtGdelk-weighted} implies
\begin{align*}
    \|\eta\,\wt\cG^\delta_k f\|_{L_{x,V}^2(H^s_t)} \le C2^{- k(\frac{d+1}{2}-s+\delta-\frac{d}{2p'})}\|f\|_{L^p(\R^d)}
\end{align*}
for $0\le s\le 1$. Taking $s>1/2$ such that $\frac{d+1}{2}-s+\delta-\frac{d}{2p'}>0$ and summing the estimate above over $k$, we derive the conclusion that $\{\eta\,\cG^\delta_{k} f\}$ is Cauchy in $L_{x,V}^2(H^s_t(\R))$.

We now prove \eqref{e:wtGdelk-weighted}. Since the $x$-support of $\eta$ is compact and $y\in\supp(\wt\chi_k)$ implies
$|y|\sim 2^k$, we have
\begin{align}\label{e:Vx-y2k}
    |U(x-y)|\sim 2^k
\end{align}
for all $(x,U)\in\supp(\eta)$, if $k$ is sufficiently large.
Hence, by the asymptotic expansion of $K^\delta$ in \eqref{i:asymKdel}, for $(x,U)\in \supp(\eta)$,
\begin{align*}
    \wt\cG_k^\delta f(x,U) & = \sum_{\pm}\sum_{\ell = 0}^{M-1}\int a_{\pm,\ell} e^{\pm i \Phi(U,x-y)} t^d |U(x-y)|^{-\frac{d+1}{2}-\delta-\ell} \wt\chi_{k}(y) f(y) dy \\
    & \qquad + \int t^d E_M(|U(x-y)|) \wt\chi_{k}(y) f(y) dy \\
    &=: \sum_{\pm}\sum_{\ell=0}^{M-1} \wt\cG_{k,\ell,\pm}^\delta f(x,U) + \wt\cG_{k,M}^\delta f(x,U),
\end{align*}
where
\begin{align*}
    \Phi(U,z) = |Uz|,\quad U\in T(d)_+,\ z\in \R^d.
\end{align*}
In view of the above decomposition, \eqref{e:wtGdelk-weighted} follows once we establish the following: for $n\in \{0,1\}$ and $0\le \ell\le M-1$,
\begin{align}\label{e:ewtcG-ell}
    \|\partial_t^n(\eta\,\wt\cG_{k,\ell,\pm}^\delta f)\|_{L^2_{x,U}} &\le C2^{-k(\frac{d+1}{2}+\ell-n+\delta -\frac{d}{2p'})}\|f\|_{L^p}, \\\label{e:ewtcG-M}
    \|\partial_t^n(\eta\,\wt\cG_{k,M}^\delta f)\|_{L^2_{x,U}} &\le C2^{-k(\frac{d+1}{2}-n+\delta -\frac{d}{2p'})}\|f\|_{L^p}.
\end{align}
It is relatively simple to prove \eqref{e:ewtcG-M}. Indeed, using \eqref{e:EMbd}, we obtain
\begin{align*}
    |E_M(|U(x-y)|)| + |\partial_t\big(E_M(|U(x-y)|)\big)| \le C(1+|U(x-y)|)^{-\frac{d-1}{2}-\delta -M},
\end{align*}
We use the fact that $E_M$ is smooth to handle the case $|U(x-y)|\le 1$.
Combining this with H\"older's inequality, we have that, for $n\in \{0,1\}$,
\begin{align*}
    &\int \big|\partial_t^n\big(\eta(x,U)\wt\cG_{k,M}^\delta f(x,U)\big)\big|^2 dxdU \\
    &\hspace{6mm} \le C\|\eta\|_{L^1} \sup_{(x,U)\in \supp(\eta)}\bigg(\int (1+|U(x-y)|)^{-(\frac{d-1}{2}+\delta +M)p'} |\wt\chi_k(y)|^{p'} dy\bigg)^{\frac{2}{p'}} \|f\|_{L^p}^2
\end{align*}
with a constant $C>0$. Since $k$ is chosen sufficiently large, \eqref{e:Vx-y2k} holds for $y\in \supp(\wt\chi_k)$ and $(x,U)\in \supp(\eta)$. Hence, choosing $M$ sufficiently large in the estimate above, we obtain \eqref{e:ewtcG-M}.

We proceed to show \eqref{e:ewtcG-ell}. Using the chain rule, we write
\begin{align*}
    &\partial_t^n(\eta(x,tV)\,\wt\cG_{k,\ell,\pm}^\delta f(x,tV)) \\
    &\qquad= \sum_{j=0}^n \eta_{j,\ell,\pm}(x,U) \int e^{i\Phi(U,x-y)} |U(x-y)|^{-\frac{d+1}{2}-\delta-\ell+j}\,\wt\chi_k(y) f(y) dy
\end{align*}
with appropriate cutoff functions $\{\eta_{j,\ell,\pm}\}_{j=0}^n\subset C_c^\infty(\R^d\times T(d)_+)$ such that 
\[
\supp(\eta_{j,\ell,\pm})\subset \supp(\eta).
\]
By \eqref{e:Vx-y2k}, it is sufficient to prove that, for every $\eta_*\in C_c^\infty(\R^d\times T(d)_+)$ such that $\supp(\eta_*) \subset \supp(\eta)$,
\begin{align}\label{e:wtGdelk-red}
    \|\cO_k[\eta_*] f\|_{L^2_{x,U}} \le C 2^{\frac{d}{2p'}k} \|f\|_{L^p(\R^d)},
\end{align}
where
\begin{align*}
    \cO_k[\eta_*] f (x,U) = \eta_*(x,U)\int e^{i\Phi(U,x-y)} \wt\chi_k(y) f(y) dy.
\end{align*}

The proof of \eqref{e:wtGdelk-red} follows as an application of a standard $T^*T$ argument. Before proceeding, we decompose $\wt\chi$ into finitely many cutoff functions so that we may assume that $\wt\chi$ is supported in a compact set $F$ of diameter $\le \ep$ with sufficiently small $\ep > 0$. We note that this support condition on $\widetilde{\chi}_k$ is assumed only in this section.

Fix $\eta_*\in C_c^\infty(\R^d\times T(d)_+)$ such that $\supp(\eta_*) \subset \supp(\eta)$. We write
\begin{align}\label{i:cOe*}
    \|\cO_k[\eta_*] f\|_{L^2_{x,U}}^2 = \int\cK^\delta(y,z) \wt\chi_k(y)\wt\chi_k(z) f(y)\overline{f(z)} dydz
\end{align}
with
\begin{align*}
    \cK^\delta(y,z) = \int e^{i(\Phi(U,x-y) - \Phi(U,x-z))} |\eta_*(x,U)|^2 dxdU.
\end{align*}
Let $K_\circ\subset \R^d$ and $K_*\subset T(d)_+$ be compact sets such that $\supp(\eta_*)\subset K_\circ\times K_*$ and $\dist(K_\circ,\, \supp(\wt\chi_k))\gtrsim 2^k$. Then, note that
\begin{align}\label{e:Kdel-standard}
    |\cK^\delta(y,z)|\le |K_\circ|\sup_{x\in K_\circ} |\cK_x^\delta(y,z)|
\end{align}
where
\begin{align*}
    \cK_x^\delta(y,z) = \int e^{i(\Phi(U,x-y) - \Phi(U,x-z))} |\eta_*(x,U)|^2 dU.
\end{align*}
Fix $x\in K_\circ$. For $U\in T(d)_+$, write $U = (u_{i,j})_{1\le i,j\le d}$, with $u_{i,j} = 0$ for $1\le j < i\le d$. Calculating the determinant of the Jacobian of the left multiplication transformation by $U$, we see that $dU = \prod_{i=1}^d u_{i,i}^{i-d-1}du_{i,i}\prod_{i<j} du_{i,j}$. Setting
\[
\wt y = 2^{-k}(x-y),\quad \wt z = 2^{-k}(x-z), 
\]
we write
\begin{align*}
    \cK_x^\delta(y,z) = \int e^{i2^k(\Phi(U,\wt y) - \Phi(U,\wt z))} \wt\eta_*(x,U) \prod_{i\le j} du_{i,j},
\end{align*}
where $\wt\eta_*(x,U) = |\eta_*(x,U)|^2 \prod_{i=1}^d u_{i,i}^{i-d-1}$. Note that $\wt y,\wt z\in \supp(\wt\chi(2^{-k}x-\cdot))$. Recall that, by the support condition of $\wt\chi_k$, 
\[
\wt y,\wt z\in 2^{-k}x - F = \{2^{-k}x-y: y\in F\}.
\]
Note that $2^{-k} x - F$ is also a compact set of diameter less than $\ep$. Since $x\in K_\circ$, $x$ is contained in a bounded set. Thus, if $k$ is sufficiently large, then $2^{-k} x - F\subset F^\circ$ for all $x$, where $F^\circ$ is a compact set whose diameter is at most $2\ep$. Identifying $T(d)_+$ with $\R_+^d\times \R^{\frac{d(d-1)}{2}}$, we denote by $\partial_U$ the differential operator
\begin{align*}
    \partial_Ug = \big(\partial_{u_{1,1}}g,\dots, \partial_{u_{1,d}}g, \partial_{u_{2,2}}g,\dots,\partial_{u_{2,d}}g,\dots,\partial_{u_{d,d}}g\big)^\intercal.
\end{align*}
To apply integration by parts in $U$, we need the following result, which asserts that the mixed Hessian $\partial_U \partial_w^\intercal\Phi(U,w)$ has maximal rank.

\begin{prop}\label{prop:rank}
    Let $U\in K_*$ and $w\in F^\circ$. Then the mixed Hessian $\partial_U \partial_w^\intercal\Phi(U,w)$ has full rank $d$. 
\end{prop}

For the moment, assume that the proposition is valid. Recall that $F^\circ$ has been chosen sufficiently small. Thus, after decomposing $K_*$ into finitely many compact patches and replacing $K_*$ by one of these patches, we may assume that one fixed $d\times d$ submatrix of $\partial_U\partial_w^\intercal\Phi(U,w)$ is uniformly invertible on $K_*\times F^\circ$. In other words, there exist indices $\{(i_m,j_m)\}_{m=1}^d$, with $i_m\le j_m$, such that \[
    \left|\det(\partial_{U'}\partial_w^\intercal\Phi(U,w))\right|\gtrsim 1
\]
for all $(U,w)\in K_*\times F^\circ$, where
\begin{align*}
    \partial_{U'}g = \big(\partial_{u_{i_1,j_1}}g,\dots, \partial_{u_{i_d,j_d}}g\big)^\intercal.
\end{align*}
From this and Taylor's theorem, it follows that
\begin{align*}
    |\partial_{U'}\Phi(U,\wt y) - \partial_{U'}\Phi(U,\wt z)| &\gtrsim |\wt y-\wt z|
\end{align*}
for $U\in K_*$ and $\wt y, \wt z\in F^\circ$, provided that the diameter of $F^\circ$ is sufficiently small.

Applying integration by parts using the operator
\begin{align*}
    L  = \frac{\partial_{U'}\Phi(U,\wt y) - \partial_{U'}\Phi(U,\wt z)}{i2^k|\partial_{U'}\Phi(U,\wt y) - \partial_{U'}\Phi(U,\wt z)|^2}\cdot \partial_{U'},
\end{align*}
we obtain
\begin{align*}
    |\cK_x^\delta(y,z)|\le C_N(1 + 2^k|\wt y-\wt z|)^{-N} = C_N(1 + |y-z|)^{-N}
\end{align*}
for $N\in \N$ with a constant $C_N>0$. Substituting the estimate above into \eqref{e:Kdel-standard} yields $|\cK^\delta(y,z)|\le C_N (1 + |y-z|)^{-N}$.
Using this together with \eqref{i:cOe*}, we get
\begin{align*}
    \|\cO_k[\eta_*] f\|_{L^2_{x,U}}^2 &\lesssim_N \Big(\sup_z\int (1+|y-z|)^{-N} |\wt\chi_k(y) f(y)| dy\Big) \int |\wt\chi_k(z) f(z)| dz \\
    &\lesssim_N 2^{\frac{d}{p'}k} \|f\|_{L^p}^2.
\end{align*}
In the second line, we use H\"older's inequality. This gives \eqref{e:wtGdelk-red}, as desired.

\subsection{Nondegeneracy of the phase} It remains to verify Proposition \ref{prop:rank}. Considering a shear transformation and modifying the sets $K_*$ and $F^\circ$, we may assume that 
\[
|w_1|\ge c
\]
for all $w\in F^\circ$, with some constant $c>0$. Indeed, since $F^\circ$ is a set of diameter $2\ep$ contained in $\{1/8\le|w|\le 4\}$, there exists $k\in \N$ and a constant $c_0>(16d)^{-1/2}$ such that $|w_k|\ge c_0$ for $w\in F^\circ$. If $|w_1|\le (64d)^{-1/2}$, we consider the upper triangular matrix 
\[
T = I + e_1e_k^\intercal
\]
and use the identity $\Phi(U,w) = \Phi(UT^{-1}, Tw)$. A straightforward calculation shows that
\begin{align*}
    \partial_U\partial_w^\intercal \Phi(U,w) = \wt T \partial_U\partial_w^\intercal \Phi(UT^{-1}, Tw) T,
\end{align*}
where $\wt T$ is an invertible $d(d+1)/2\times d(d+1)/2$ matrix. Hence, Proposition \ref{prop:rank} is equivalent to the statement that $\partial_U\partial_w^\intercal \Phi(UT^{-1}, Tw)$ has rank $d$ for all $U\in K_*$ and $w\in F^\circ$. Moreover, $UT^{-1}\in T(d)_+$ for $U\in K_*$, and 
\[
|\inp{Tw}{e_1}| = |w_1+w_k|\ge (64d)^{-1/2}
\]
for all $w\in F^\circ$. Thus, replacing $K_*$ and $F^\circ$ by $\{UT^{-1}:U\in K_*\}$ and $\{Tw:w\in F^\circ\}$, respectively, we may assume that $|w_1|\ge (64d)^{-1/2}$ on $F^\circ$.

Choose $U\in K_*$, $w\in F^\circ$, and fix them. For $1\le k\le d$, let $\bfu_k$, $\bfw_k$ denote vectors in $\R^d$ given by
\begin{align*}
    \bfu_k = (0,\dots,0,u_{k,k},\dots,u_{k,d})^\intercal,\quad \bfw_k = \begin{cases}
        w & k=1, \\
        w-\sum_{\ell < k} w_\ell e_\ell & 2\le k\le d.
    \end{cases}
\end{align*}
Here, $\bfu_1$ denotes $(u_{1,1},\dots, u_{1,d})$. Moreover, let $\bfu'_k = (0,\dots,0,u_{k,k+1},\dots, u_{k,d})$ for $1\le k\le d-1$. We now define for $1\le k\le d$ 
\begin{align*}
    U^{[1]}_k = \bfw_k\bfu_k^\intercal + \bfw_k^\intercal \bfu_k \wt I_k,\quad \bfv^{[1]}_k = \bfw_k^\intercal \bfu_k \bfw_k
\end{align*}
where $\wt I_k$ is a $d\times d$ matrix given as the block matrix
\begin{align*}
    \wt I_k = \begin{pmatrix}
        0 & 0 \\
        0 & I_{d+1-k}
    \end{pmatrix}.
\end{align*}
Furthermore, for $2\le k\le d$, let $\wt U_k^{[1]}$ denote a $(d+1-k)\times d$ matrix obtained by deleting $k-1$ top rows of $U_k^{[1]}$. Similarly, we define $\wt \bfv_k^{[1]}$ for $2\le k\le d$. Then, a calculation shows that
\begin{align*}
    \partial_U\partial_w^\intercal \Phi(U,w) = (\Phi(U,w))^{-1} \begin{pmatrix}
        & \wt U_1^{[1]} & \\
        & \vdots & \\
        & \wt U_d^{[1]} & 
    \end{pmatrix} \, - (\Phi(U,w))^{-3} \begin{pmatrix}
        & \wt \bfv_1^{[1]} & \\
        & \vdots & \\
        & \wt \bfv_d^{[1]} & 
    \end{pmatrix}\bfb,
\end{align*}
where
\begin{align*}
    \bfb = \frac12\partial_w^\intercal\big((\Phi(U,w))^2\big).
\end{align*}
For $1\le k\le d$, let $S^{[1]}_k$ denote the augmented matrix given by
\begin{align*}
    S^{[1]}_k = \left(\begin{array}{ccc|c}
         & & &  \\
         & U^{[1]}_k & & \bfv_k^{[1]} \\
         & & &
    \end{array}\right).
\end{align*}
Note that, for $2\le k\le d$, the $k-1$ top rows of $S^{[1]}_k$ are identically zero. Similarly, $\inp{\bfw_k}{e_\ell} = 0$ for $1\le \ell\le k-1$ if $2\le k\le d$. Let $M^{[1]}$ be a $d^2\times d$ matrix given by
\begin{align*}
    M^{[1]} = (\Phi(U,w))^{-1} \begin{pmatrix}
        & U_1^{[1]} & \\
        & \vdots & \\
        & U_d^{[1]} & 
    \end{pmatrix} \, - (\Phi(U,w))^{-3} \begin{pmatrix}
        & \bfv_1^{[1]} & \\
        & \vdots & \\
        & \bfv_d^{[1]} & 
    \end{pmatrix}\bfb.
\end{align*}
We note that $M^{[1]}$ is obtained from $\partial_U\partial_w^\intercal \Phi$ by inserting $d(d-1)/2$ additional zero rows, and thus $\rank(\partial_U\partial_w^\intercal \Phi) = \rank(M^{[1]})$. Hence the proposition is equivalent to the statement that $M^{[1]}$ has rank $d$.
To establish this, we first prove Lemma \ref{lem:canonical} below, which will significantly simplify the matrix $M^{[1]}$. To state the result, for $1\le k\le d$, let $V_k = \bfw_ke_k^\intercal + w_k \wt I_k$ and
\begin{align*}
    G_k = \left(\begin{array}{ccc|c}
         & & &  \\
         & V_k & & w_k\bfw_k \\
         & & &
    \end{array}\right).
\end{align*}

\begin{lemma}\label{lem:canonical}
    Suppose that $U\in T(d)_+$ and $w\in \R^d$. Then there exists a collection $\{E_{k,\ell}\}_{1\le k\le \ell\le d}$ of $d\times d$ matrices such that for $1\le k\le d$, $E_{k,k}$ is invertible and
    \begin{align}\label{i:Gk-linES}
        G_k = \sum_{k\le \ell\le d}E_{k,\ell}S_\ell^{[1]}.
    \end{align}
\end{lemma}

Assuming the lemma for the moment, we complete the proof of Proposition \ref{prop:rank}. By Lemma \ref{lem:canonical}, there exists an invertible $d^2\times d^2$ matrix $E$ given in block form by
\begin{align*}
    E = \begin{pNiceMatrix}[cell-space-limits = 2pt]
        E_{1,1} & E_{1,2} & E_{1,3} & \cdots & E_{1,d} \\
        0& E_{2,2} & E_{2,3} & \cdots & E_{2,d} \\
        0 & 0 & E_{3,3} & \cdots & E_{3,d} \\
        \vdots & \vdots & & \ddots & \vdots \\
        0 & 0 & 0 & \cdots & E_{d,d}
    \end{pNiceMatrix}
\end{align*}
such that
\begin{align*}
    \Phi(U,w) EM^{[1]} = \begin{pmatrix}
        & V_1 & \\
        & \vdots & \\
        & V_d & 
    \end{pmatrix} \, - (\Phi(U,w))^{-2} \begin{pmatrix}
        & w_1\bfw_1 & \\
        & \vdots & \\
        & w_d\bfw_d & 
    \end{pmatrix}\bfb^\intercal.
\end{align*}
Hence, the proof would be complete once we prove that the matrix
\begin{align*}
    V_1 - w_1(\Phi(U,w))^{-2}\bfw_1\bfb^\intercal
\end{align*}
is invertible. Recall that $|w_1| > c$ for a constant $c>0$, which implies that $V_1$ is invertible. By a direct calculation,
\begin{align*}
    V_1 - w_1(\Phi(U,w))^{-2}\bfw_1\bfb^\intercal = V_1\big(I_d - \frac12(\Phi(U,w))^{-2}\bfw_1\bfb^\intercal\big).
\end{align*}
On the other hand, we note that $\bfb^\intercal \bfw_1 = (\Phi(U,w))^2$. Therefore, by the matrix determinant lemma, we have
\begin{align*}
    \det(V_1 - w_1(\Phi(U,w))^{-2}\bfw_1\bfb^\intercal) &= \det(V_1)\big(1 - \frac12 (\Phi(U,w))^{-2}\bfb^\intercal\bfw_1\big) \\
    &=\frac12\det(V_1)\neq 0,
\end{align*}
as desired.

\begin{proof}[Proof of Lemma \ref{lem:canonical}]
    We construct $E_{k,\ell}$ inductively on $k$ in decreasing order. In the case $k=d$, which corresponds to the base case, we note
    \begin{align}\label{i:Udvds}
        U^{[1]}_d = 2w_du_{d,d}e_de_d^\intercal = u_{d,d}V_d,\quad \bfv_d^{[1]} = w^2_d u_{d,d} e_d = u_{d,d}w_d\bfw_d.
    \end{align}
    We use that $\wt I_d = e_de_d^\intercal$. Hence letting
    \begin{align*}
        E_{d,d} = \sum_{1\le \ell\le d-1} e_\ell e_\ell^\intercal +  u_{d,d}^{-1} e_de_d^\intercal,
    \end{align*}
    we see that $E_{d,d}$ satisfies the assertions of the lemma for $k=d$.

    Let $1\le m\le d-1$. We now assume that the matrices $E_{k,\ell}$ are constructed for all $m+1\le k\le \ell\le d$ so that $E_{k,k}$ are invertible and \eqref{i:Gk-linES} holds for such $k$. We will construct matrices $\{E_{m,\ell}\}_{m\le \ell\le d}$ fulfilling the required assertions. To this end, for $1\le k\le d-1$, we let
    \begin{align*}
        U_k^{[2]} &= 2 u_{k,k} w_k e_ke_k^\intercal + \sum_{k+1\le \ell\le d} (u_{k,k}e_\ell + u_{k,\ell}e_k)(w_k e_\ell + w_\ell e_k)^\intercal, \\
        \bfv_k^{[2]} &= w_k \bfw_k^\intercal \bfu_k e_k + u_{k,k} w_k \bfw_{k+1},
    \end{align*}
    and define $S^{[2]}_k$ analogously to $S_k^{[1]}$.
    By computation, we then have
    \begin{align}\label{i:Sm[1]-m+1&[2]}
        S_m^{[1]} = S^{[2]}_m + S_{m+1}^{[1]}[\bfu_m'],
    \end{align}
    where $S_{m+1}^{[1]}[\bfu_m']$ denotes a $d\times d$ matrix defined by replacing $u_{m+1, m+1}, \dots, u_{m+1,d}$ in the definition of $S_{m+1}^{[1]}$ with $u_{m,m+1},\dots, u_{m,d}$, respectively. We define $S_{m+1}^{[1]}[\bfu]$ for general vectors $\bfu\in \R^d$ in the same manner. If $m=d-1$, by \eqref{i:Udvds},
    \begin{align}\label{i:Sm+1[um']}
        S_{m+1}^{[1]}[\bfu_m'] = \frac{u_{d-1,d}}{u_{d,d}} S_{d}^{[1]}.
    \end{align}
    In the case $m\le d-2$, to make the factor $u_{m,m+1}$ positive, we choose a large constant $C = C(U)>0$ so that $u_{m,m+1} + Cu_{m+1,m+1} >0$. Observe that the induction hypothesis for $k=m+1$ can be applied to $S_{m+1}^{[1]}[\bfu'_m + C\bfu_{m+1}]$ because $\inp{\bfu'_m + C\bfu_{m+1}}{e_{m+1}}>0$. As a consequence, there exist $d\times d$ matrices $\{\wt E_{m+1,\ell}\}_{m+1\le \ell\le d}$ such that
    \begin{align}\label{i:Gm+1-wtEs}
        G_{m+1} = \wt E_{m+1, m+1}S_{m+1}^{[1]}[\bfu'_m + C\bfu_{m+1}] + \sum_{m+2\le \ell\le d} \wt E_{m+1,\ell} S_\ell^{[1]}
    \end{align}
    and $\wt E_{m+1, m+1}$ is invertible. On the other hand, the induction hypothesis also implies that there exist $d\times d$ matrices $\{ E_{m+1,\ell}\}_{m+1\le \ell\le d}$ such that
    \begin{align*}
        G_{m+1} = \sum_{m+1\le \ell\le d} E_{m+1,\ell} S_\ell^{[1]}
    \end{align*}
    with $E_{m+1, m+1}$ invertible. Combining this with \eqref{i:Gm+1-wtEs}, we obtain
    \begin{align}
    \begin{aligned}\label{i:Sm+1[um']-msmall}
        S_{m+1}^{[1]}[\bfu'_m] &= \big(\wt E_{m+1,m+1}^{-1}E_{m+1,m+1}-CI\big)S_{m+1}^{[1]} \\
        &\hspace{24mm}+ \sum_{m+2\le \ell\le d} \wt E_{m+1,m+1}^{-1}\big(E_{m+1,\ell} - \wt E_{m+1,\ell}\big) S_\ell^{[1]}
    \end{aligned}
    \end{align}
    because $S_{m+1}^{[1]}[\bfu'_m + C\bfu_{m+1}] = S_{m+1}^{[1]}[\bfu'_m] + CS_{m+1}^{[1]}$.

    We turn to the task of estimating $S_m^{[2]}$ on the right-hand side of \eqref{i:Sm[1]-m+1&[2]}. For $1\le k\le d-1$, let $L_k$ be a $(d+1-k)\times (d+1-k)$ lower triangular matrix defined by
    \begin{align*}
        L_k = u_{k,k}^{-2}\begin{pNiceMatrix}[cell-space-limits = 3pt]
            u_{k,k} & 0 & 0 & 0 & 0 & 0 \\
            -u_{k,k+1} & u_{k,k} & 0 & 0 & 0 & 0 \\
            -u_{k,k+2} & 0 & u_{k,k} & 0 & 0 & 0 \\
            \vdots & \vdots & 0 & u_{k,k} &  & \vdots \\
            \vdots & \vdots & \vdots & & \ddots & \vdots \\
            -u_{k,d} & 0 & 0 & \cdots & \cdots & u_{k,k}
        \end{pNiceMatrix}.
    \end{align*}
    We now define the $d\times d$ matrix $E_{m,m}$ by
    \begin{align*}
        E_{m,m} = \begin{pmatrix}
            I_{m-1} & 0 \\
            0 & L_m^\intercal
        \end{pmatrix} \quad \text{if}\ \ 2\le m\le d-1,
    \end{align*}
    and $E_{m,m} = L_m^\intercal$ if $m=1$. A calculation shows that, for $m+1\le \ell\le d$, 
    \begin{align*}
        E_{m,m} e_m = u_{m,m}^{-1} e_m,\quad E_{m,m} (u_{m,m}e_\ell + u_{m,\ell}e_m) = e_\ell.
    \end{align*}
    Thus, we get
    \begin{align*}
        E_{m,m} U_m^{[2]} & = \sum_{m\le \ell\le d} e_\ell (w_me_\ell + w_\ell e_m)^\intercal = V_m, \\
        E_{m,m} \bfv_m^{[2]} & = w_m^2 e_m + w_m \bfw_{m+1} = w_m\bfw_m, 
    \end{align*}
    implying that $E_{m,m} S_m^{[2]} = G_m$. By definition, $E_{m,m}$ is invertible, thus $S_m^{[2]} =E_{m,m}^{-1} G_m$. Combining this with \eqref{i:Sm[1]-m+1&[2]}, \eqref{i:Sm+1[um']}, and \eqref{i:Sm+1[um']-msmall}, we deduce that
    \begin{align}\label{i:Gm-mlarge}
        G_m = E_{m,m} S_m^{[1]} -\frac{u_{d-1,d}}{u_{d,d}} E_{m,m}S_d^{[1]}
    \end{align}
    if $m=d-1$, and
    \begin{align}\label{i:Gm-msmall}
        G_m = E_{m,m} S_m^{[1]} + \sum_{m+1\le \ell \le d} E_{m,m}\mathring E_{m,\ell}S_\ell^{[1]},
    \end{align}
    if $m\le d-2$, where
    \begin{align*}
        \mathring E_{m,m+1} = CI - \wt E_{m+1,m+1}^{-1}E_{m+1,m+1},\quad \mathring E_{m,\ell} = -\wt E_{m+1,m+1}^{-1}\big(E_{m+1,\ell} - \wt E_{m+1,\ell}\big)
    \end{align*}
    for $m+2\le \ell\le d$. Define the matrices $E_{m,\ell}$, for $m+1\le \ell\le d$, by 
    \begin{align*}
        E_{m,\ell} = \begin{cases}
            -\dfrac{u_{d-1,d}}{u_{d,d}}E_{m,m} & \text{if}\ \ \ell = d\ \text{and}\  m=d-1, \\
            E_{m,m}\mathring E_{m,\ell} & \text{if}\ \ m\le d-2.
        \end{cases}
    \end{align*}
    This completes the proof of the lemma.
\end{proof}

\subsubsection{Remark} An examination of the proof shows that Proposition \ref{prop:convergence} is independent of the choice of $\wt\chi_k$ in the definition of $\wt\cG_k^\delta f$. Indeed, our argument works as long as $\wt\chi$ vanishes near $0$.

\section{Implication from Theorem \ref{thm:define-xtU} to Theorem \ref{thm:define-xt}}\label{sec:implication}

In this section, we deduce Theorem \ref{thm:define-xt} from Theorem \ref{thm:define-xtU}. We first assert that
\begin{align*}
    S^\delta f\in \cD'(\R^d\times \R_+)\quad\text{for every $f\in L^p(\R^d)$}.
\end{align*}
This follows by the same argument as in the proof of Theorem \ref{thm:define-xtU}, using Lemma \ref{lem:integrable}. We omit the details. 

Since
\begin{align}\label{r:inclusion}
    L^2_{x,loc}(H^s_{t,loc})\subset L^1_{x,loc}(H^s_{t,loc})
\end{align}
for $s>1/2$, it is sufficient to prove the following stronger assertion: for $f\in L^p(\R^d)$, we have $f_V\in \wt\cC_{p,\delta}$ for Haar-a.e. $V\in ST(d)_+$, where
\begin{align*}
    \wt\cC_{p,\delta} = \{f\in L^p: S^\delta f\in L^2_{x,loc}(H^s_{t,loc})\ \text{for some}\ s>\frac12\}.
\end{align*}
To this end, we begin by obtaining the following result.

\begin{lemma}\label{lem:identity1}
    Let $f\in L^p(\R^d)$. Then, for $\varphi_\circ \in C_c^\infty(\R^d\times \R_+)$ and $\varphi_*\in C_c^\infty (ST(d)_+)$ we have
    \begin{align}\label{i:identity1}
        \int \Big(S^\delta f_V(\varphi_\circ) - \int \cT^\delta f(V^{-1}x,tV)\varphi_\circ (x,t) dxdt\Big) \varphi_*(V) dV = 0.
    \end{align}
    %Moreover, the function $\fT^\delta f(x,t,V) := \cT^\delta f_V(x,tI)$ is in $L^2_{x,V,loc}(H^s_{t,loc})$.
\end{lemma}

\begin{proof}
    Let $\varphi_\circ\in C_c^\infty(\R^d\times \R_+)$ and $\varphi_*\in C_c^\infty(ST(d)_+)$. Choose
    \begin{align*}
        \varphi(x,t,V) = t\varphi_\circ(Vx,t) \varphi_*(V).
    \end{align*}
    Clearly, $\varphi\in C_c^\infty(\R^d\times T(d)_+)$. Applying a change of variables $x\to V^{-1}x$ and using the fact that $dU = t^{-1}dtdV$, we have
    \begin{align}\label{i:cTdel-id}
        \mathcal T^\delta f(\varphi) &= \int \bigg(\int t^d K^\delta (t(x-Vy)) \varphi_\circ(x,t) \varphi_*(V) dxdtdV \bigg) f(y) dy
    \end{align}
    for all $f\in L^p(\R^d)$. We also use $\det(V) = 1$.
    An application of Lemma \ref{lem:integrable} shows that
    \begin{align*}
        \int \bigg|\int t^d K^\delta (t(x-Vy)) \varphi_\circ(x,t) dxdt\bigg|^{p'} dy < \infty,
    \end{align*}
    which enables us to interchange two integrals over $V$ and $y$ in \eqref{i:cTdel-id}. Changing variables $y \to V^{-1}y$, we write
    \begin{align}\nonumber
        \mathcal T^\delta f(\varphi) &= \int \bigg(\int \bigg(\int t^d K^\delta (t(x-y)) \varphi_\circ(x,t) dxdt \bigg) f_V(y) dy\bigg) \varphi_*(V) dV \\\label{i:cTdelta1}
        &= \int S^\delta f_{V}(\varphi_\circ) \varphi_*(V) dV.
    \end{align}
    On the other hand, using Theorem \ref{thm:define-xtU} and making a change of variables $x\to V^{-1}x$ we write
    \begin{align*}
        \cT^\delta f(\varphi) = \int \cT^\delta f(V^{-1}x,tV) \varphi_\circ(x, t)\varphi_*(V)  dxdtdV.
    \end{align*}
    Subtracting it from \eqref{i:cTdelta1} we get \eqref{i:identity1}, as desired.
\end{proof}

By Lemma \ref{lem:identity1}, we see that, for $f\in L^p(\R^d)$ and $\varphi_\circ\in C_c^\infty(\R^d\times\R_+)$, there exists $E_{f,\varphi_\circ}\subset ST(d)_+$ such that $dV((E_{f,\varphi_\circ})^\complement) = 0$ and
\begin{align}\label{i:SdelfV}
    S^\delta f_V(\varphi_\circ) = \int \cT^\delta f(V^{-1}x,tV)\varphi_\circ (x,t) dxdt
\end{align}
for all $V\in E_{f,\varphi_\circ}$. We will modify the set $E_{f,\varphi_\circ}$ so that it is independent of $\varphi_\circ$. Since $C_c^\infty(\R^d\times \R_+)$ is separable, there is a dense countable collection $\{\varphi_k\}_{k\in \N}\subset C_c^\infty(\R^d\times \R_+)$. On the other hand, for $V\in ST(d)_+$ and $f\in L^p$, let
\begin{align*}
    \fT^\delta_V f(x,t) = \cT^\delta f(V^{-1}x, tV).
\end{align*}
Since $\cT^\delta f\in L_{x,V,loc}^2(H^s_{t,loc})$, there exists $\wt F \subset ST(d)_+$ such that $dV(\wt F^\complement) = 0$ and $\fT^\delta_V f\in  L^2_{x,loc}(H_{t,loc}^s)$ for all $V\in \wt F$. For $f\in L^p(\R^d)$, we define
\begin{align*}
    F_f = \wt F \cap \Big(\bigcap_{k=1}^\infty E_{f,\varphi_k}\Big).
\end{align*}
Since $(E_{f,\varphi_\circ})^\complement$ and $\wt F^\complement$ are sets of measure zero,
\begin{align}\label{i:measF_fc}
    dV(ST(d)_+\setminus F_f) = 0.
\end{align}
We claim that for $f\in L^p(\R^d)$, \eqref{i:SdelfV} holds for all $V\in F_f$ and $\varphi_\circ\in C_c^\infty(\R^d\times \R_+)$. Since $\{\varphi_k\}$ is a dense subset of $C_c^\infty(\R^d\times \R_+)$, it suffices to prove that the two linear functionals
\begin{align*}
    \varphi_\circ \mapsto S^\delta f_V(\varphi_\circ),\ \ \varphi_\circ \mapsto\int \fT^\delta_V f(x,t)\varphi_\circ (x,t) dxdt
\end{align*}
are continuous for $V\in F_f$. The continuity of the first mapping follows from $S^\delta f_V\in \cD'(\R^d\times \R_+)$. The second linear functional above is bounded because $\fT^\delta_V f\in  L^2_{x,loc}(H_{t,loc}^s)$ for $V\in \wt F$. Hence, the linear functionals above are continuous, proving the claim.

To complete the proof, we claim
\begin{align*}
    \bigcup_{f\in L^p} \{f_V: V\in F_f\}\subset \wt\cC_{p,\delta}.
\end{align*}
By definition, Theorem \ref{thm:define-xt} immediately follows once the claim is shown. To prove the claim, let $g\in \bigcup_{f\in L^p} \{f_V: V\in F_f\}$. Using the definition, we see that $g = f_V$ with some $f\in L^p$ and $V\in F_f$, thus 
\begin{align*}
    S^\delta g (\varphi_\circ) = \int \fT^\delta_V f(x,t) \varphi_\circ (x,t) dxdt
\end{align*}
for all $\varphi_\circ\in C_c^\infty(\R^d\times \R_+)$. It implies that $S^\delta g = \fT_V^\delta f$ in $\cD'(\R^d\times \R_+)$. From the previous argument, $\fT^\delta_V f \in L^2_{x,loc}(H^s_{t,loc})$ with $s>1/2$. Hence, $S^\delta g \in L^2_{x,loc}(H^s_{t,loc})$, and therefore, $g\in \wt\cC_{p,\delta}$, as desired.

\section{Failure of almost everywhere convergence for $\delta < \delta(d,p)$}\label{sec:failure}
The objective of this section is to establish Theorem \ref{thm:failure}. Throughout this section, we assume that $\delta < \delta(d,p)$.
%As discussed in the introduction, we construct a collection of functions satisfying the conditions listed above.
Let $\psi$ be a Schwartz function on $\R$ such that
\[
\supp(\widehat{\psi}\hspace{0.2mm})\subset (1/4,2),\quad \widehat{\psi}\equiv 1\ \text{on}\ [1/2,1],\quad \psi(0)>0.
\]
For $\ep>0$, we denote $N = N(\ep) = \ep^{-\gamma}$, where $\gamma$ is a positive number such that $\gamma<10^{-10}$ and 
\begin{align}\label{i:cond-gamma}
    \delta - \delta(d,p) + \gamma\,\frac{\delta(d,p)}{2} < 0.
\end{align}
Also, we define
\[
f_\ep (x) := \psi(\ep^{-1}N^{-1}(N+\Delta))(x,0) = \widecheck{\mathcal F_\ep}(x),
\]
for $\ep>0$, where $\mathcal F_\ep(\xi) := \psi(\ep^{-1}N^{-1}(N-|\xi|^2))$.
For simplicity, we write
\[
\ep_j := 2^{-2^j},\quad N_j := N(\ep_j) = \ep_j^{-\gamma},\quad f_j(x) := f_{\ep_j}(x),\quad \mathcal F_j(x) := \mathcal F_{\ep_j}(x)
\]
for $j\ge 0$. Also, we define
\[
F(x) := \sum_{j\ge 0} 2^{-j}\,\frac{f_j(x)}{\|f_j\|_{L^p}}.
\]
From the triangle inequality, it clearly follows that $F\in L^p(\R^d)$.
In the subsequent subsections, we prove that $F$ is the desired counterexample.

\subsection{Preliminary estimates}
We first obtain preparatory estimates for $f_\ep$, which will serve as useful tools in the proof of Theorem \ref{thm:failure}.
The following lemma shows that $f_\ep$ is rapidly decaying away from
the region $|x|\sim \ep^{-1}N^{-1/2}$.

\begin{lemma}\label{lem:decay}
    Let $\ep>0$. Then the following holds.
    \begin{enumerate}[topsep=1pt, itemsep=1pt]
        \item [i)] For $M\in \N$, there exists a constant $C = C(d,M)>0$ such that
        \begin{align}\label{e:fedecay}
            |f_\ep(x)|\le C\, \begin{cases}
    \ep^M (N\ep)^{\frac d2} & \text{ if } |x|\ll \ep^{-1}N^{-\frac12} \\
    (N\ep |x|^2)^{-M} (N\ep)^{\frac d2} & \text{ if } |x|\gg \ep^{-1}N^{-\frac12}.
\end{cases}
        \end{align}
    \item [ii)] If $|x|\sim \ep^{-1}N^{-1/2}$, then we have
    \[
    |f_\ep(x)|\lesssim \ep^{\frac12}(N\ep)^{\frac d2}.
    \]
    \end{enumerate}
\end{lemma}

\begin{proof}
    Using the Fourier inversion formula, we note that
    \[
    \psi(\ep^{-1}N^{-1}(N+\Delta)) = c\int \widehat{\psi}(t) e^{i\ep^{-1}N^{-1}t(N+\Delta)} dt, \quad c\in\C.
    \]
    Combining this with the explicit kernel formula $e^{it\Delta }(x,y) = (4\pi it)^{-d/2} e^{i|x-y|^2/4t}$, we deduce that $f_\ep(x)$ is expressed by the integral
    \begin{align}\label{id:ferep}
        c\,(\ep N)^{\frac d2} \int \widehat{\psi}(t) t^{-\frac d2} e^{i\phi_{\ep,N}(t,x)} dt,
    \end{align}
    where $c\in\C$ and
    \[
    \phi_{\ep, N}(t,x) := \frac{t}{\ep} + \frac{N\ep|x|^2}{4t}.
    \]
    Now let us assume either $|x|\ll \ep^{-1}N^{-\frac12}$ or $|x|\gg \ep^{-1}N^{-\frac12}$.
    Recall that $\supp(\widehat{\psi}\hspace{0.2mm})\subset (1/4,2)$.
    By a routine calculation, one can easily check that the bounds
    \[
    |\partial_t\phi_{\ep, N}(t,x)| \gtrsim \begin{cases}
        \ep^{-1}, & |x|\ll \ep^{-1}N^{-\frac12}, \\
        N\ep|x|^2, & |x|\gg \ep^{-1}N^{-\frac12}
    \end{cases}
    \]
    and
    \[
    |\partial_t^n\phi_{\ep, N}(t,x)| \lesssim N\ep|x|^2
    \]
    hold for any $t\in \supp(\widehat{\psi}\hspace{0.2mm})$, $n\ge 2$.
    Integration by parts based on these bounds shows that $f_\ep(x)$ satisfies the bound \eqref{e:fedecay}.
    To verify $ii)$, assume that $|x|\sim \ep^{-1}N^{-1/2}$.
    In the case $\partial_t\phi_{\ep, N}$ may be zero, but for the second derivative in $t$ we see that
    \[
    |\partial_t^2 \phi_{\ep, N}(t,x)| = \frac{N\ep |x|^2}{2t^3} \gtrsim \ep^{-1}
    \]
    on $\supp(\widehat{\psi}\hspace{0.2mm})$.
    Applying van der Corput's lemma \cite[p.332]{St93} to \eqref{id:ferep} yields the estimate in \textit{ii)}, as desired.
\end{proof}

Utilizing Lemma \ref{lem:decay}, we obtain sharp $L^p$ bounds of $f_\ep$ as follows.

\begin{lemma}
    Let $2\le p< \infty$. For $\ep>0$, we have
    \begin{align}\label{e:lowfe}
        \|f_\ep\|_{L^p} \sim \ep^{\frac12}(N\ep)^{\frac d2}(\ep^{-1}N^{-1/2})^{\frac dp}.
    \end{align}
\end{lemma}
\begin{proof}
    We define
    \begin{align*}
        A^1_{\ep} &:= \{x\in \R^d : |x|\le 10^{-10} \ep^{-1}N^{-\frac12}\}, \\
        A^2_{\ep} &:= \{x\in \R^d : 10^{-10} \ep^{-1}N^{-\frac12}\le |x|\le 10^{10} \ep^{-1}N^{-\frac12}\}, \\
        A^3_{\ep} &:= \{x\in \R^d : 10^{10} \ep^{-1}N^{-\frac12}\le |x|\}.
    \end{align*}
    Then by triangle inequality and Lemma \ref{lem:decay}, we get
    \begin{align*}
        \|f_\ep\|_{L^p} &\le \|f_\ep\|_{L^p(A^2_\ep)} + \|f_\ep\|_{L^p(A^1_\ep)} + \|f_\ep\|_{L^p(A^3_\ep)} \\
        & \lesssim_M (N\ep)^{\frac d2}\bigg(\ep^{\frac12} |A^2_\ep|^{\frac1p} + \ep^M |A^1_\ep|^{\frac1p} + \Big(\int_{A^3_\ep} (N\ep|y|^2)^{-Mp} dy\Big)^{\frac 1p} \bigg) \\
        &\lesssim_M \ep^{\frac12} (N\ep)^{\frac d2} (\ep^{-1}N^{-\frac12})^{\frac dp},
    \end{align*}
    which gives the upper bound in the lemma.
    To obtain the lower bound, using Lemma \ref{lem:decay}, we see that
    \begin{align}\nonumber
        \|f_\ep\|_{L^p} &\ge \|f_\ep\|_{L^p(A_\ep^2)} - (\|f_\ep\|_{L^p(A_\ep^1)}+\|f_\ep\|_{L^p(A_\ep^3)}) \\\label{e:lowfep}
        &\gtrsim_M \|f_\ep\|_{L^p(A_\ep^2)} - \ep^M(N\ep)^{\frac d2} (\ep^{-1}N^{-1/2})^{\frac dp}.
    \end{align}
    To estimate $\|f_\ep\|_{L^p(A_\ep^2)}$, we apply H\"older's inequality together with Lemma \ref{lem:decay} to get
    \begin{align}\notag
        \|f_\ep\|_{L^p(A_\ep^2)}&\ge |A_\ep^2|^{\frac1p-\frac12} \|f_\ep\|_{L^2(A_\ep^2)} \\\notag
        & \gtrsim (\ep^{-1}N^{-\frac12})^{\frac dp-\frac d2} \big(\|f_\ep\|_{L^2} - \|f_\ep\|_{L^2(A^1_\ep)} - \|f_\ep\|_{L^2(A_\ep^3)}\big) \\\label{e:lowfep2}
        & \gtrsim (\ep^{-1}N^{-\frac12})^{\frac dp-\frac d2} \big(\|f_\ep\|_{L^2} - \ep^MN^{\frac d4}\big).
    \end{align}
    Recall that $f_\ep = \widecheck{\mathcal F_\ep}$, thus by the Plancherel theorem $\|f_\ep\|_2 = \|\mathcal F_\ep\|_2$.
    Since $\psi(0) > 0$, there exists $c>0$ such that $|\psi(t)|\ge c$ for $|t|\le c$. 
    From this, we see that
    \begin{align*}
        \|\mathcal F_\ep\|_2 &\ge \bigg(\int_{\{\ep^{-1}N^{-1}|N-|\xi|^2|\le\,c\}} \big|\psi(\ep^{-1}N^{-1}(N-|\xi|^2))\big|^2 d\xi\bigg)^{\frac12} \\
        &\ge c\,\big|\{\ep^{-1}N^{-1}|N-|\xi|^2|\le c\}\big|^{\frac12} \\
        &\gtrsim \ep^{\frac12} N^{\frac d4}.
    \end{align*}
    Combined with \eqref{e:lowfep2}, this gives the estimate 
    \[
    \|f_\ep\|_{L^p(A_\ep^2)}\gtrsim \ep^{\frac12}(N\ep)^{\frac d2}(\ep^{-1}N^{-1/2})^{\frac dp}.
    \]
    Substituting this into \eqref{e:lowfep} and taking $M$ sufficiently large, the lower bound in \eqref{e:lowfe} follows, as required.
\end{proof}

\subsection{Control of $S_t^\delta F$}
The aim of this subsection is to show that $S_t^\delta F$ remains bounded when $t$ is restricted to a bounded interval. Recall the cutoff functions $\chi_{\circ, k}$ defined in Section \ref{sec:thmxtU}. For $k\in \N$, we let
\begin{align*}
    \cH_k^\delta F(x,t) = \int t^d K^\delta(t(x-y)) \chi_{\circ, k}(y) F(y) dy.
\end{align*}

\begin{prop}\label{prop:welldef}
Let $K\subset \R^d$ be compact and let $t_0>0$. Then, the following holds.
\begin{itemize}[topsep=1pt, itemsep=1pt]
    \item [i)] $\cH_k^\delta F(x,t)$ and $\partial_t\cH_k^\delta F(x,t)$ converge uniformly on $K\times (0,t_0]$ as $k\to \infty$.
    \item [ii)] There exists a constant $C = C(K,t_0)>0$ such that
    \begin{align*}
        \sup_{k\in \N} \, \sup_{(x,t)\in K\times (0,t_0]} \, \big(|\cH^\delta_k F(x,t)| + |\partial_t \cH^\delta_k F(x,t)|\big) \le C.
    \end{align*}
\end{itemize}
\end{prop}
Using the proposition, we can show that $F\in \cC_{p,\delta}$. Indeed, by $i)$ of the proposition, we have that $\cH_k^\delta F$ is a Cauchy sequence in $L^2_{x,loc}(H^s_{t,loc})$ for $1/2 < s< 1$. This follows from the interpolation inequality
\begin{align*}
    \|\eta (\cH_k^\delta F - \cH_{k'}^\delta F)\|_{L^2_{x}(H^s_{t})} \le C \big(\|\cH_k^\delta F - \cH_{k'}^\delta F\|_{L^\infty(\wt K)} + \|\partial_t(\cH_k^\delta F - \cH_{k'}^\delta F)\|_{L^\infty(\wt K)}\big),
\end{align*}
for $1/2<s<1$ with a constant $C =C(s,\eta,K)>0$, where $\eta\in C_c^\infty(\R^d\times \R_+)$ and $\wt K\subset \R^d\times \R_+$ is a compact set such that $\supp(\eta) \subset \wt K$.
Moreover, by a similar argument in Section \ref{sec:thmxtU},
\[
S^\delta F = \lim_{k\to \infty} \cH_k^\delta F \quad \text{in}\ \cD'(\R^d\times \R_+).
\]
Combining this with the fact that $\cH_k^\delta F$ is Cauchy, we conclude that 
\[
S^\delta F\in L^2_{x,loc}(H^s_{t,loc})
\]
for $1/2 < s < 1$, implying $F\in \wt\cC_{p,\delta}$. By \eqref{r:inclusion}, we conclude $F\in \cC_{p,\delta}$.
Furthermore, by $ii)$ of the proposition, $S_t^\delta F$ is bounded on $(0,t_0]\times K$ for every compact set $K\subset \R^d$.

We now prove Proposition \ref{prop:welldef}. In the proof, we assume that $x\in K$ for some compact set $K$ and $t\le t_0$. From this assumption, it follows that there exists a positive constant $\wt C$ such that $|x|\le \wt C$. 

First, note that $ii)$ of the proposition is straightforward if the supremum is taken over $k\le 10$ rather than over $k\in \N$. Indeed,
\begin{align}\label{e:tdKdel-t^d+1}
    |t^d K^\delta(t(x-y))| + |\partial_t(t^d K^\delta(t(x-y)))|\le C (1 + t)^{d}(1+|x-y|)
\end{align}
for all $0<t\le t_0$ and $\delta$, with $C>0$ depending only on $d,\delta$ and $t_0$. To see this, we use the trivial bound
\[
|\partial_t K^\delta(t(x-y))| \le (2\pi)^{-d}\int \big|\big(1-|\xi|^2\big)_+^\delta e^{ it\inp{x-y}{\xi}} \inp{x-y}{\xi}\big| d\xi \lesssim 1+|x-y|
\]
and the identity 
\[
\partial_t(t^dK^\delta(tz)) = dt^{d-1}K^\delta(tz) + t^d\nabla K^\delta(tz)\cdot z,\quad z = x-y.
\]
Combining \eqref{e:tdKdel-t^d+1} with Hölder's inequality yields
\begin{align*}
    |\cH_k^\delta F(x,t)| + |\partial_t \cH_k^\delta F(x,t)|\le C t_0^{d+1} \|\chi_{\circ, k}\|_{L^{p'}}\|F\|_{L^p} \le C
\end{align*}
for all $x\in K$, $t\le t_0$, and all $k\le 10$. This proves the desired bound for $k\le 10$.

Observe that, for $k>10$, $\cH_k^\delta F$ is written as the telescoping sum
\begin{align*}
    \cH_k^\delta F(x,t) = \cH_{10}^\delta F(x,t) + \sum_{10\le\ell\le k-1} (\cH_{\ell+1}^\delta F(x,t) - \cH_{\ell}^\delta F(x,t)).
\end{align*}
Thus, to show the proposition, it is sufficient to verify that, for every $k\in \N$,
\begin{align}\label{id:welldef}
    \sup_{(x,t)\in K\times (0,t_0]} \big(\big|\wt\cH_k^\delta F(x,t)\big| + \big|\partial_t\wt\cH_k^\delta F(x,t)\big|\big) \le C 2^{-ck}
\end{align}
with constants $c, C > 0$ depending on $K$ and $t_0$, where
\begin{align*}
    \wt\cH_k^\delta F(x,t) = \int t^d K^\delta(t(x-y)) \wt\chi_k(y) F(y) dy,\quad k\in \N.
\end{align*}
To this end, fix $(x,t)\in \R^d\times (0,t_0]$ and $k\in \N$. We also let $n\in \{0,1\}$. Using the triangle inequality and \eqref{e:lowfe}, we note
\begin{align}\label{e:stdsum}
    |\partial_t^n\wt\cH_k^\delta F(x,t)|\lesssim \sum_{j\ge 0} 2^{-j} \ep_j^{-\frac12}(N_j\ep_j)^{-\frac d2}(\ep_j N_j^{\frac12})^{\frac dp} |\partial_t^n\wt\cH_k^\delta f_j (x,t)|,
\end{align}
where $\partial_t^0$ denotes the identity operator.
We claim that for $M\in\N$, there is a constant $C = C(d,t_0,M)>0$ such that
\begin{align}\label{e:bd}
    |\partial_t^n\wt\cH_k^\delta f_\ep(x,t)|\le C (N\ep)^{\frac d2} 2^{k(d+1)} \begin{cases}
         \ep^M , & 2^k\ll \ep^{-1}N^{-\frac12}, \\
        N^{-\frac M2} , & 2^k\sim \ep^{-1}N^{-\frac12}, \\
         (N\ep 2^{2k})^{-M} , & 2^k\gg \ep^{-1}N^{-\frac12}.
    \end{cases}
\end{align}
Assuming for the moment that \eqref{e:bd} holds, we finish the proof of \eqref{id:welldef}.
For $k\gg 1$, we define sets of indices
\begin{align*}
    \cI^k_1 &:= \{j\in \N_0 : \ep_j^{-1}N_j^{-\frac12}\le 10^{-10}2^k\}, \\
    \cI^k_2 &:= \{j\in \N_0 : 10^{-10}2^k < \ep_j^{-1}N_j^{-\frac12}\le 10^{10}2^k\}, \\
    \cI^k_3 &:= \{j\in \N_0 : 10^{10}2^k < \ep_j^{-1}N_j^{-\frac12}\},
\end{align*}
where $\N_0 = \N\cup \{0\}$.
Note that $|\cI_2^k|\le C$ for some universal constant $C$, and that $\cI_2^k$ may be empty.
It follows from \eqref{e:bd} that each summand in \eqref{e:stdsum} is controlled by
\begin{align}\label{b:stbsuand}
    C\, 2^{-j} \ep_j^{-\frac12} (\ep_j N_j^{\frac12})^{\frac dp} 2^{k(d+1)}\,\begin{cases}  (N_j\ep_j 2^{2k})^{-M} & \text{ if } j\in \cI^k_1, \\
     N_j^{-\frac M2} & \text{ if } j\in \cI^k_2, \\
     \ep_j^M & \text{ if } j\in \cI^k_3,
    \end{cases}
\end{align}
with some $C = C(d,t_0,M)>0$.
Since $N_j\ge 1$ for $j\in \N$ and
\begin{align*}
    \ep_j^{-1}N_j^{-\frac12} \le 10^{-10}2^k\ \ \ \text{if}\ j\in \cI_1^k,\quad \ep_jN_j^{\frac12} < 10^{-10}2^{-k}\ \ \ \text{if}\ j\in \cI_3^k,
\end{align*}
the bounds in \eqref{b:stbsuand} for the first and third cases are bounded by
\begin{align*}
    C\,2^{-j} 2^{-K_1 k},\quad K_1 = \frac dp - d -\frac32 + M
\end{align*}
for some $C>0$.
On the other hand, we note
\[
\ep_j \sim 2^{-\frac{2k}{2-\gamma}},\quad N_j \sim 2^{\frac{2\gamma k}{2 - \gamma}},\quad \text{if \ $j\in \cI_2^k$}.
\]
Hence, the bound in \eqref{b:stbsuand} for the second case can be expressed as
\begin{align*}
    C\, 2^{-j} 2^{-K_2 k},\quad K_2 = \frac dp -d-1 +\frac{\gamma M-1}{2-\gamma}.
\end{align*}
for some $C>0$. Now we choose $M$ sufficiently large such that $K_1, K_2 >1$.
Combining all of these, we get
\begin{align}\nonumber
    \big|\partial_t^n\wt \cH_k^\delta F(x,t)\big|&\le C\bigg( \sum_{j\in \cI_1^k\cup \cI_3^k} 2^{-K_1 k} 2^{-j} + \sum_{j\in \cI_2^k} 2^{-K_2 k} 2^{-j}\bigg) \\\label{b:stdbd}
    &\le C\,\big(2^{-K_1 k} + 2^{-K_2 k}\big)
\end{align}
with $C = C(d,t_0,M) > 0$.
It is obvious that the bound in the last line converges to $0$ as $k$ tends to infinity.
Therefore, \eqref{id:welldef} follows.

\begin{proof}[Proof of \eqref{e:bd}]
    We first address the case $n=1$. The bounds for both cases $2^k\ll \ep^{-1}N^{-1/2}$ and $2^k\gg \ep^{-1}N^{-1/2}$ are straightforward consequences of Lemma \ref{lem:decay}.
    Indeed, by \eqref{e:tdKdel-t^d+1},
    \begin{align*}
        |\partial_t\wt\cH^\delta_k f_\ep(x,t)|\le C t_0^d\,2^{k(d+1)} \sup_{y\,\in\, \supp(\wt\chi_k)} |f_\ep(y)|.
    \end{align*}
    Applying \textit{i)} in Lemma \ref{lem:decay} shows that $\partial_t\wt\cH^\delta_k f_{\ep}$ satisfies the required bounds in \eqref{e:bd}. We also use that $\wt\chi(y)$ vanishes unless $1/4\le |y|\le 4$.
    
    We turn to proving the estimate for the case $2^k\sim \ep^{-1}N^{-1/2}$. After rescaling, we write
    \[
    \wt\cH^\delta_k f_\ep(x,t) = (2\pi)^{-d}t^d\int \int \big(1-|\xi|^2\big)_+^\delta e^{it\inp{x-y}{\xi}} \wt\chi_k(y) f_\ep(y) dyd\xi.
    \]
    Replacing $f_\ep$ with the term \eqref{id:ferep}, we express the right-hand side as
    \[
    c(\ep N)^{\frac d2}t^d\int \int \big(1-|\xi|^2\big)_+^\delta \cH_{k,\ep}(s,t\xi) \widehat{\psi}(s) s^{-\frac d2} e^{i(t\inp{x}{\xi}+\ep^{-1} s)} d\xi ds
    \]
    with $c\in \C$ and the function $\cH_{k,\ep}$ defined by
    \[
    \cH_{k,\ep}(s,\xi) = \int e^{-i\inp{y}{\xi}} \wt\chi_k(y) e^{i\frac{N\ep|y|^2}{4s}}dy.
    \]
    Using the above expression, we calculate
    \begin{align*}
        \partial_t\wt\cH^\delta_k f_\ep(x,t) &= c(\ep N)^{\frac d2}\int \int \big(1-|\xi|^2\big)_+^\delta \widehat{\psi}(s) s^{-\frac d2} e^{i(t\inp x\xi +\ep^{-1}s)} \\
        &\qquad\times \big(d t^{d-1}\cH_{k,\ep}(s,t\xi) + it^d\inp x\xi \cH_{k,\ep}(s,t\xi) + t^d\partial_t \cH_{k,\ep}(s,t\xi)\big) d\xi ds.
    \end{align*}
    The variable $\xi$ in the integrand is bounded by $1$. Combining this with the assumptions $t\le t_0$ and $|x|\le \wt C$, we deduce that for any $(x,t)\in K$,
    \begin{align}\label{e:stdbd2}
        |\partial_t \wt \cH^\delta_k f_\ep(x,t)| \le C (\ep N)^{\frac d2} t_0^d \sup_{s\,\in\, \supp(\widehat{\psi}\hspace{0.2mm}), |\xi|\le 1} \big(|\cH_{k,\ep}(s,t\xi)| + |\partial_t\cH_{k,\ep}(s,t\xi)|\big),
    \end{align}
    with $C>0$ independent of $x,t$. By a calculation,
    \begin{align*}
        \partial_t\cH_{k,\ep}(s,t\xi) = -\int i\inp y\xi e^{-it\inp{y}{\xi}} \wt\chi_k(y) e^{i\frac{N\ep|y|^2}{4s}} dy.
    \end{align*}
    To analyze this integral, we define the operator
    \[
    L_{\ep, s} f(y) := \bigg(i\frac{N\ep d}{2s}-\frac{N^2\ep^2 |y|^2}{4s^2}\bigg)^{-1} \Delta_y f(y).
    \]
    Then it can be seen that $L_{\ep,s} e^{i\frac{N\ep|y|^2}{4s}} = e^{i\frac{N\ep|y|^2}{4s}}$. 
    Thus integration by parts using the operator $L_{\ep,s}$ gives the identity
    \[
    \partial_t\cH_{k,\ep}(s,t\xi) = -i\int (L_{\ep,s}^*)^M\big(\inp y\xi e^{-it\inp{y}{\xi}}\wt\chi_k(y)\big) e^{i\frac{N\ep|y|^2}{4s}} dy.
    \]
    A routine calculation, using that $y\in \supp(\wt\chi_k)$ implies $|y|\sim 2^k$, shows that
    \[
    \bigg|\partial_y^\alpha \bigg(i\frac{N\ep d}{2s}+\frac{N^2\ep^2 |y|^2}{4s^2}\bigg)^{-1} \bigg| \lesssim_{\alpha} N^{-1} (\ep N^{\frac12})^{|\alpha|}
    \]
    for any $s\in\supp(\widehat{\psi}\hspace{0.2mm})$ and $\alpha\in\N_0^d$.
    Combining this with the trivial bound
    \[
    \big|\partial_y^\alpha\big(\inp y\xi e^{-it\inp{y}{\xi}} \wt\chi_k(y)\big)\big|\lesssim_\alpha 2^k\big(|\xi|^2 + 2^{-2k}\big)^{\frac{|\alpha|}{2}},\quad \alpha\in\N_0^d,\ |\xi|\le 1
    \]
    yields
    \begin{align}\nonumber
        \big|\partial_t \cH_{k,\ep}(s,t\xi)\big| &\le C 2^k\int_{\supp(\wt\chi_k)} \big((|\xi|^2 + 2^{-2k} + \ep^2 N) N^{-1}\big)^M dy \\\label{e:bdcHkep}
        &\le C (1+t)^{2M} N^{-M} 2^{k(d+1)},
    \end{align}
    with a constant $C= C(d,t_0,M)>0$ whenever $|\xi|\le 1$ and $s\in \supp(\widehat{\psi}\hspace{0.2mm})$. The last inequality follows from the fact that $\ep^2 N$ becomes less than $1$, because $\gamma < 10^{-10}$.
    A similar argument also shows $|\cH_{k,\ep}(s,\xi)|\le C (1+t)^{2M} N^{-M} 2^{kd}$.
    Substituting this and \eqref{e:bdcHkep} into \eqref{e:stdbd2} gives the second estimate in \eqref{e:bd}. 
    
    It only remains to deal with the case $n=0$, in which the proof follows the same argument as before. The problem becomes simpler since the term $\partial_t \cH_{k,\ep}(s,\xi)$ does not appear in this case. We omit the details.
\end{proof}

\subsection{Failure of almost everywhere convergence}
In this subsection, we prove that $S_t^\delta F(x)$ diverges as $t\to\infty$ on a set of positive measure.
By $ii)$ of Proposition \ref{prop:welldef}, the problem is reduced to that of proving the following.
\begin{prop}
    Assume that $\delta < \delta(d,p)$. Then we have
    \[
    \big|\{x\in \R^d : |x|\le 10,\ \sup_{t>0}\hspace{0.2mm}|S_t^\delta F(x)| = \infty\}\big| \gtrsim 1.
    \]
\end{prop}
Throughout the proof, we assume $|x|\le 10$. To proceed, we consider $\chi_-^\nu\in \mathcal S'(\R)$ defined by
\[
\inp{\chi_-^\nu}{g} = \frac{\inp{x_-^\nu}{g}}{\Gamma(\nu + 1)} ,\quad g\in \mathcal S(\R)
\]
for $\text{Re}\,\nu > -1$, where $\Gamma$ denotes the Gamma function and $x_-^\nu := |x|^\nu \chi_{(-\infty, 0)}(x)$, $\nu\in \C$.
Although $x_-^\nu$ is not locally integrable when $\text{Re}\,\nu \le -1$, we can define $\chi_-^\nu$ as a tempered distribution for such $\nu$'s by using the analytic continuation argument, see \cite[Section 3.2]{H83}.
Now we define for $\Psi\in C_c^\infty([0,\infty))$ and $\nu\in \R$, the Weyl fractional derivative $\Psi^{(\nu)}$ of $\Psi$ of order $\nu$ by $\Psi^{(\nu)} = \Psi * \chi_-^{-\nu-1}$. It can be written as
\[
\Psi^{(\nu)}(t) = \frac{c_\nu}{\Gamma(-\nu)}\int_0^\infty u^{-\nu-1} \Psi(t+u) du 
\]
with $c_\nu\in \C$ if $\nu<0$. For $\nu\ge 0$,
\begin{align}\label{id:fnu}
    \Psi^{(\nu)}(t) = \frac{c_{\nu,\nu_0}}{\Gamma(-\nu+\nu_0)}\int_0^\infty u^{-\nu+\nu_0-1} \Psi^{(\nu_0)}(t+u) du
\end{align}
with $c_{\nu,\nu_0}\in \C$, where $\nu_0$ is a natural number such that $\nu<\nu_0\le \nu+1$ and $\Psi^{(\nu_0)}$ appearing in the integral is interpreted as a standard $\nu_0$-th derivative of $\Psi$.
Then it is known that $\Psi^{(\nu)}$ satisfies the relation $\Psi = \Psi^{(\nu)}*\chi_-^{\nu-1}$.
Using this, we write
\begin{align}\nonumber
    \Psi(-\Delta) F(x) &= (2\pi)^{-d}\lim_{k\to \infty}\int \Psi(|\xi|^2)\,e^{i\inp{x-y}{\xi}} \chi_{\circ, k}(y) F(y) d\xi dy \\\nonumber
    &= (2\pi)^{-d}\lim_{k\to \infty}\int \bigg(\int_0^\infty \Psi^{(\delta+1)}(t) \big(|\xi|^2 - t\big)_-^\delta dt\bigg) e^{i\inp{x-y}{\xi}} \chi_{\circ, k}(y) F(y) d\xi dy \\\label{id:fdelta}
    &=(2\pi)^{-d}\lim_{k\to \infty} c_\delta\int_0^\infty t^\delta \Psi^{(\delta+1)}(t)
    \cH_k^\delta F(x,\sqrt t)\,dt
\end{align}
with $c_\delta\in \C$.
To justify passing the limit inside the integral, we use \eqref{id:fnu}, which shows that $\Psi^{(\delta+1)}$ is an integrable function whose support is compact. Combining this with \eqref{id:welldef} yields
\begin{align}\nonumber
    \big|\Psi(-\Delta) F(x)\big| &= \frac{1}{\Gamma(\delta+1)} \bigg| \int_0^\infty t^\delta \Psi^{(\delta+1)}(t)
    S_{\sqrt t}^\delta F(x)\,dt\bigg| \\\label{e:fdest}
    &\lesssim_\delta\,\sup_{t>0}\big|S_t^\delta F(x)\big| \int_0^\infty \big|t^\delta \Psi^{(\delta+1)}(t)\big|dt.
\end{align}
Now we set $\Psi_j(t) := \eta((N_j\ep_j)^{-1}(t-N_j))$ for each $j\ge 0$, where $\eta\in C_c^\infty(\R)$ such that $\supp(\eta)\subset (-c_0,c_0)$, $\eta\ge 0$, $\eta\equiv1$ on $[-c_0/2,c_0/2]$. Here $c_0>0$ is chosen so small that 
\begin{align}\label{c:eta-repsi}
    \operatorname{Re}\psi(s)\ge c>0\quad \text{ for $|s|\le c_0$}.
\end{align}
Replacing $\Psi$, $\delta$ appearing in \eqref{id:fnu} with $\Psi_j$ and $\delta+1$, respectively, we get
\[
\Psi_j^{(\delta+1)}(t) = C_\delta (N\ep)^{-\delta_1} \int_0^\infty u^{-\delta+\delta_1-2} \eta^{(\delta_1)}((N\ep)^{-1} (t+u-N)) du.
\]
Here $\delta_1$ denotes a natural number such that $\delta+1 < \delta_1 \le \delta+2$. 
By change of variables $u' = (N\ep)^{-1} (t+u-N)$, the right-hand side equals
\begin{align}\label{t:nepint}
    C_\delta (N\ep)^{-\delta-1} \int_{(N\ep)^{-1}(t-N)}^\infty \big(u' + (N\ep)^{-1}(N-t)\big)^{-\delta+\delta_1-2} \eta^{(\delta_1)}(u')du',
\end{align}
which is bounded by
\begin{align}\label{t:bd-nepint}
    C_\delta (N\ep)^{-\delta-1} \big(1 + (N\ep)^{-1}|N - t|\big)^{-\delta-2} \chi_{(-\infty, N+1]}(t).
\end{align}
Indeed, if $t>N(1+2\ep)$, then the integral in \eqref{t:nepint} vanishes because 
\[
\supp(\eta)\cap ((N\ep)^{-1}(t-N),\infty) = \emptyset.
\]
When $N(1-2\ep)\le t\le N(1+2\ep)$, \eqref{t:nepint} is bounded by a constant times $(N\ep)^{-\delta-1}$. This follows from $(N\ep)^{-1}|N-t|\le 2$ and $-1 < -\delta+\delta_1-2\le 0$, ensuring that $(u'+(N\ep)^{-1}|N-t|)^{-\delta + \delta_1 - 2}$ is locally integrable near $0$.
Lastly, in the case $t<N(1-2\ep)$, we make use of integration by parts $\delta_1$ times to write \eqref{t:nepint} as
\[
C_\delta (N\ep)^{-\delta-1} \int \big(u' + (N\ep)^{-1}(N-t)\big)^{-\delta-2} \eta(u')du'.
\]
The desired bound \eqref{t:bd-nepint} follows immediately from this expression.

Combining \eqref{t:bd-nepint} with \eqref{e:fdest}, we deduce
\begin{align*}
    \big|\Psi_j(-\Delta)F(x)\big|\lesssim_\delta (N\ep)^{-\delta-1} \sup_{t>0}\big|S_t^\delta F(x)\big|\int_0^{N+1} t^\delta \big(1 + (N\ep)^{-1}|N - t|\big)^{-\delta-2} dt.
\end{align*}
By a simple calculation, we conclude that $\Psi_j(-\Delta)F$ satisfies the bound
\begin{align}\label{e:fjstd}
    \big|\Psi_j(-\Delta) F(x)\big| \lesssim \ep^{-\delta} \sup_{t>0}\big|S_t^\delta F(x)\big|,
\end{align}
where we omit the dependence on $\delta$. Now we claim that there exists a collection of sets $\{\cA_j\}_{j\in\N}$ contained in $\{1/2\le |x|\le 2\}$ such that for all $j$ sufficiently large,
\begin{align*}
    &\big|\Psi_j(-\Delta) f_j(x)\big|\ge C \big(\ep_j N_j^{\frac12}\big)^{-\delta(d,p)}\|f_j\|_{L^p}\ \ \text{for}\ x\in \cA_j, \\
    &\big|\cA_j\big|\ge C
\end{align*}
with a fixed constant $C>0$ independent of $j$.
Postponing the proof of the claim until the end of this subsection and assuming this for the moment, we complete the proof.
Since $\cA_j\subset \{1/2\le |x|\le 2\}$ and $|\cA_j|\ge C$ for all $j$, 
\[
\big|\limsup_{j\to \infty} \cA_j\big|\ge C.
\]
This is easy to see. Indeed, from the bounds of $\cA_j$ it follows that $|\bigcup_{j=\ell}^\infty \cA_j|\ge C$ for any $\ell$. Since $\limsup \cA_j = \bigcap_{\ell\ge 1} \bigcup_{j=\ell}^\infty \cA_j$, the conclusion follows.
Now we let $x\in \limsup \cA_j$. Then we can choose a subsequence $\{j^x\}\subset \N$ satisfying 
\begin{align*}
    \big|\Psi_{j^x}(-\Delta) f_{j^x}(x)\big|\ge C (\ep_{j^x} N_{j^x}^{\frac12})^{-\delta(d,p)}\|f_{j^x}\|_{L^p}.
\end{align*}
Using this together with \eqref{e:fjstd} gives
\begin{align}\nonumber
    \sup_{t>0}\big|S_t^\delta F(x)\big| &\gtrsim \ep_{j^x}^\delta \big|\Psi_{j^x}(-\Delta) F(x)\big| \\\label{e:stdlb}
    &\gtrsim 2^{-j^x}\ep_{j^x}^{\delta - \delta(d,p)} N_{j^x}^{-\frac{\delta(d,p)}{2}} - \ep_{j^x}^{\delta}\bigg(\sum_{k\neq j^x} 2^{-k}\|f_{k}\|_{L^p}^{-1}\big|\Psi_{j^x}(-\Delta) f_{k}(x)\big|\bigg).
\end{align}
In the last line, the second summand is negligible compared to the first one if $j^x$ is sufficiently large.
To see this, we note
\begin{align}\label{i:Psifk-int}
    \Psi_{j^x}(-\Delta) f_k(x) = \int \Psi_{j^x}(|\xi|^2) \mathcal F_k(\xi) e^{i\inp{x}{\xi}} d\xi.
\end{align}
Since 
\[
|\mathcal F_k|\le C_M\begin{cases}
    \ep_{j^x}^M, & k< j^x, \\
    \ep_k^M, & k>j^x
\end{cases}
\]
for any $M\in\N$ on the support of $\Psi_{j^x}(|\xi|^2)$, \eqref{i:Psifk-int} implies
\begin{align*}
    \big|\Psi_{j^x}(-\Delta) f_{k}(x)\big|%&\lesssim \ep_{j^x}^{M} \int \big|\eta\big((N_{j^x}\ep_{j^x})^{-1}(N_{j^x} - |\xi|^2)\big)\big|\,d\xi \\
    &\lesssim_M \begin{cases}
        \ep_{j^x}^{-10d+M}, & k<j^x \\
        \ep_k^{-10d+M}, & k>j^x.
    \end{cases}
\end{align*}
Substituting this into \eqref{e:stdlb} and applying \eqref{e:lowfe} gives
\begin{align*}
    &\ep_{j^x}^{\delta}\bigg(\sum_{k\neq j^x} 2^{-k}\|f_k\|_{L^p}^{-1}\big|\Psi_{j^x}(-\Delta) f_k(x)\big|\bigg) \\
    &\hspace{32mm} \lesssim \ep_{j^x}^{\delta}\bigg(\sum_{k< j^x} 2^{-k}\ep_{j^x}^{-100d+M} + \sum_{k> j^x} 2^{-k}\ep_{k}^{-100d+M}\bigg) \\
    &\hspace{32mm}\ll 2^{-j^x}\ep_{j^x}^{\delta - \delta(d,p)} N_{j^x}^{-\frac{\delta(d,p)}{2}}
\end{align*}
if we choose $M>1000d$, as desired.
Now we observe that the first summand in \eqref{e:stdlb} tends to infinity as $j^x\to \infty$ since \eqref{i:cond-gamma} implies
\[
2^{-j^x}\ep_{j^x}^{\delta - \delta(d,p)} N_{j^x}^{-\frac{\delta(d,p)}{2}} = 2^{-j^x} 2^{-\sigma 2^{j^x}},\quad \sigma = \delta -\delta(d,p) +\gamma\frac{\delta(d,p)}{2} < 0.
\]
Therefore, $\sup_{t>0}\big|S_t^\delta F(x)\big| = \infty$ for every $x\in \limsup \cA_j$.
This proves the proposition.

\begin{proof}[Proof of the claim]
    Using polar coordinates, we write
    \begin{align*}
        \Psi_{j}(-\Delta) f_{j}(x) &= \int \Psi_j(r^2)\wt{\mathcal F}_{j}(r) \int_{\mathbb S^{d-1}} e^{i\inp{rx}{\sigma}}d\sigma_{\mathbb S^{d-1}} r^{d-1} dr\\
        &= c \int \Psi_j(r^2)\wt{\mathcal F}_{j}(r) J_{(d-2)/2}(r|x|) r^{\frac d2} |x|^{\frac{2-d}{2}}dr,\quad c>0,
    \end{align*}
    where $J_m(r)$ denotes the Bessel function of order $m$ and $\wt{\mathcal F}_{j}\in C_c^\infty(\R)$ defined by
    \[
    \wt{\mathcal F}_{j}(r) := \psi((\ep_j N_j)^{-1}(N_j - r^2)).
    \]
    Recall that $J_m(r)$ follows the asymptotic behavior
    \[
    J_m(r) = \bigg(\frac{2}{\pi}\bigg)^{\frac12} r^{-\frac12} \cos\Big(r-\frac{\pi m}{2}-\frac{\pi}{4}\Big) + O(r^{-\frac 32})
    \]
    when $r$ is large, see \cite[p.338]{St93} for example. Hence $\Psi_j(-\Delta)f_{j}(x)$ is expressed as
    \begin{align}\label{t:fjfej}
        c \int \Psi_j(r^2)\wt{\mathcal F}_{j}(r) r^{\frac {d-1}{2}} |x|^{\frac{1-d}{2}} \cos\Big(r|x|-\frac{d-1}{4}\pi\Big) dr + E,
    \end{align}
    where $E$ is an error term satisfying the bound
    \begin{align}\label{e:bder}
        |E|\lesssim \int \Psi_j(r^2)\wt{\mathcal F}_{j}(r) r^{\frac {d-3}2} dr \lesssim N_j^{\frac{d-1}{4}}\ep_j
    \end{align}
    provided that $1/2\le |x|\le 2$, because $\Psi_j(r^2)\wt{\mathcal F}_{j}(r)$ is supported in $\{N_j^{1/2}(1-4\ep_j)\le r\le N_j^{1/2}(1+4\ep_j)\}$.
    Now we set
    \[
    \cA_j := \Big\{x : 1/2\le |x|\le 2,\ \cos\Big(N_j^{\frac12}|x|-\frac{d-1}{4}\pi\Big)>\frac12 \Big\}
    \]
     for each $j$.
     Clearly, $\big|\cA_j\big|\ge C$ for some $C>0$ independent of $j$.
     Also, if $N_j$ is sufficiently large, $\cos\big(r|x|-(d-1)\pi/4\big)>1/4$ for all $x\in \cA_j$ and $r$ in the support of $\Psi_j(r^2)\wt{\mathcal F}_{j}(r)$.
     Therefore, by \eqref{c:eta-repsi}, the real part of the leading term in \eqref{t:fjfej} is bounded below by $C N_j^{(d+1)/4}\ep_j$ on $\cA_j$.
     Combining this and \eqref{e:bder} together with \eqref{e:lowfe}, we have
     \begin{align*}
         \big|\Psi_j(-\Delta)f_{j}(x)\big|&\gtrsim \ep_j N_j^{\frac{d+1}{4}} \ep_j^{-\frac{d+1}{2}+\frac dp} N_j^{-\frac d2+\frac{d}{2p}}\|f_{j}\|_{L^p} \\
         &\gtrsim \big(\ep_j N_j^{\frac12}\big)^{-\delta(d,p)}\|f_{j}\|_{L^p}
     \end{align*}
     for $x\in \cA_j$, which is the desired result.
\end{proof}

\subsection*{Acknowledgement}
This work was supported by the National Research Foundation of Korea(NRF) grant funded by the Korea government(MSIT) (No. RS-2026-25483775). The author thanks Eunhee Jeong, Sanghyuk Lee, and Andreas Seeger for their valuable comments and helpful discussions.

\end{document}